\documentclass[11pt]{amsart}
\usepackage[a4paper, top=3cm, bottom=3cm, left=4cm, right=4cm]{geometry}
\usepackage{amsmath}
\usepackage{amssymb}
\usepackage{amscd}
\usepackage{comment}
\usepackage{url}
\usepackage[all]{xypic}

\usepackage{tikz}
\usepackage[]{hyperref}
\usepackage{graphicx}
\usepackage{todonotes}
\newcommand{\bp}{\begin{pmatrix}}
\newcommand{\ep}{\end{pmatrix}}
\newcommand{\be}{\begin{equation}}
\newcommand{\ee}{\end{equation}}
\newcommand{\ol}[1]{\overline{#1}}

\numberwithin{equation}{section}

\theoremstyle{plain}
\newtheorem{theorem}[equation]{Theorem}
\newtheorem*{theorem*}{Theorem}
\newtheorem{lemma}[equation]{Lemma}
\newtheorem{proposition}[equation]{Proposition}

\newtheorem{corollary}[equation]{Corollary}

\theoremstyle{definition}
\newtheorem{example}[equation]{Example}

\newtheorem{definition}[equation]{Definition}

\theoremstyle{remark}
\newtheorem{remark}[equation]{Remark}

\newtheorem*{ack}{Acknowledgements}

\theoremstyle{plain}

\numberwithin{equation}{section}

\newcommand{\intfrac}[2]{\genfrac{\lfloor}{\rfloor}{}{1}{#1}{#2}}
\newcommand{\rintfrac}[2]{\genfrac{\lceil}{\rceil}{}{1}{#1}{#2}}

\def\Z{\mathbb Z}
\def\R{\mathbb R}
\def\Q{\mathbb Q}
\def\C{\mathbb C}
\def\F{\mathbb F}
\def\N{\mathbb N}

\def\wt#1{\widetilde{#1}}

\def\CP{\C P}

\def\cfk{\mathit{CFK}}
\def\cf{\mathit{CF}}
\def\cfi{\mathit{CFI}}
\def\hfi{\mathit{HFI}}
\def\hfk{\mathit{HFK}}
\def\hf{\mathit{HF}}

\def\und{\underline{d}}
\def\ovd{\overline{d}}
\def\unV{\underline{V}}
\def\ovV{\overline{V}}
\newcommand{\sss}{\ifmmode{{\mathfrak s}}\else{${\mathfrak s}$\ }\fi}
\newcommand{\sst}{\ifmmode{{\mathfrak t}}\else{${\mathfrak t}$\ }\fi}
\def\spinc{Spin$^c$}
\def\di{\partial^{\iota}}

\DeclareMathOperator\iim{Im}
\makeatletter
\let\@wraptoccontribs\wraptoccontribs
\makeatother

\title[Rational cuspidal curves]{Involutive Heegaard Floer homology and rational cuspidal curves}
\author{Maciej Borodzik}
\address{Institute of Mathematics, Polish Academy of Science, ul. \'Sniadeckich 8, 00-656  Warsaw, Poland}
\address{Institute of Mathematics, University of Warsaw, ul. Banacha 2,
02-097 Warsaw, Poland}
\email{mcboro@mimuw.edu.pl}
\thanks{The second author was supported by NSF grants DMS-1128155, DMS-1307879, DMS-1552285, and a Sloan Research Fellowship.}

\author{Jennifer Hom}
\address {School of Mathematics, Georgia Institute of Technology, Atlanta, GA 30332}
\email{hom@math.gatech.edu}

\contrib[with an appendix by]{Andrzej Schinzel}
\address{Institute of Mathematics, Polish Academy of Science, ul. \'Sniadeckich 8, 00-656  Warsaw, Poland}
\email{schinzel@impan.pl}

\makeatletter
\@namedef{subjclassname@2010}{\textup{2010} Mathematics Subject Classification}
\makeatother
\subjclass[2010]{} 

\begin{document}
\begin{abstract}
We use invariants of Hendricks and Manolescu coming from involutive Heegaard Floer theory to find constraints on 
possible configurations
of singular points of a rational cuspidal curve of odd degree in the projective plane. We show that the results do not carry over
to rational cuspidal curves of even degree.
\end{abstract}
\maketitle

\section{Introduction}
\subsection{Overview}
Rational cuspidal curves are complex algebraic curves  that are homeomorphic to $S^2$. From an algebraic point
of view a rational cuspidal curve is an algebraic curve of genus zero (that is, it is rational), 
all of whose singular points have one branch (that is, they
are cuspidal). In the article we will discuss only rational cuspidal curves in $\C P^2$. 

Rational cuspidal curves have been an object of interest for a long time. There are many conjectures and open problems on the subject.
The ultimate problem is to classify all rational cuspidal curves, a task with only a few existing partial results, e.g., 
\cite{Bodn,FLMN04,Fen,Fent,FlZa95}.
Other open problems include bounding the maximal number of singular points of a rational cuspidal curve: the
strongest bound of 6 is due to Palka \cite{Pal-fin}, and there is a conjecture of Orevkov that a rational cuspidal curve can have at most 4
singular points; see \cite{Piontkowski} for a precise statement. 
Another problem is the Flenner--Zajdenberg rigidity conjecture; see for instance \cite{FlZa95}, recent advances
related to this conjecture are discussed in \cite{Pal-fin}.

Recently rational cuspidal curves have drawn a renewed attention. On the one hand new algebraic methods have been developed by Koras and
Palka. These methods, based on the minimal model program, have lead to a solution of the Cooligde--Nagata conjecture \cite{KP,Pal}
as well as the Zajdenberg finiteness conjecture \cite{Pal-fin}. (The latter conjecture was reproved by Orevkov in \cite{Or-proof} using \cite{Tono05}.) 
There is a work in progress on giving a full classification, at least
under the rigidity conjecture; see \cite{PP} for details.

Another turning point was the paper of Fernandez de Bobadilla, Luengo, Melle-Hernandez and N\'emethi \cite{FLMN06}, which brought modern methods of
low-dimensional topology into the realm of rational cuspidal curves. 
Namely, based on the study of Seiberg--Witten invariants of links of surface singularities, the authors of \cite{FLMN06} stated
a conjecture on the Alexander polynomials of links of singularities of rational cuspidal curves. 
The solution in \cite{BL} to this conjecture
(the result of \cite{BL} is slightly different than the original conjecture)
revealed further connections between rational cuspidal curves, lattice homology and the Ozsv\'ath--Szab\'o $d$--invariants.
The main result of \cite{BL} was given a new perspective in a paper by Bodn\'ar and N\'emethi \cite{BodN}.
 Another, more precise, conjecture,
based on connections with lattice homology, was also proposed in \cite{BodN}; it remains open.

\subsection{Main results}
In the present paper we follow the approach of \cite{BL}, but we use another method, namely involutive Heegaard Floer theory,
defined by Hendricks and Manolescu in \cite{HM} and based on previous work of Manolescu \cite{Man}. As explained in Section~\ref{sec:differences}, the result we obtain is essentially different than the Bodn\'ar--N\'emethi conjecture \cite{BodN}. 

Before we state the first main result, recall that for a singularity with one branch the $\delta$--invariant is the three--genus
of the link of the singularity. The singular point  has Puiseux sequence $(p;q)$ if and only if the link of singularity is
a torus knot $T(p,q)$. A \emph{regular continued fraction expansion} $\frac{q}{p}=[a_0,\ldots,a_k]$ is a continued fraction expansion
\begin{equation}\label{eq:contfracexp}
\frac{q}{p}=[a_0,a_1,\ldots,a_k]=a_0+\cfrac{1}{a_1+\cfrac{1}{a_2+\ldots+\frac{1}{a_k}}}
\end{equation}
such that 
\begin{itemize}
	\item $a_i$ are integers,
	\item for $i>0$, $a_i$ is positive,
	\item $a_k \geq 2$.
\end{itemize}
With this notation in place, we are ready to state our first result.

\newtheorem*{thm:main2}{Theorem \ref{thm:main2}}
\begin{thm:main2}
Suppose $C$ is a rational cuspidal curve with singular points $z_1,\ldots,z_n$, $n>1$. Assume that $\deg C$ is odd.
Let $\delta_1,\ldots,\delta_n$ be the 
$\delta$--invariants. 
Assume that $z_1$ has Puiseux sequence $(p;q)$, and that $\delta_1\in S_1$, where $S_1$ is the semigroup of $z_1$. Write the regular
continued fraction
$\frac{q}{p}=[a_0,a_1,\ldots,a_k]$ with $a_k > 1$.
Then $\delta_2+\ldots+\delta_n>\intfrac{a_k-1}{2}$.
\end{thm:main2}
\begin{remark}
The assumption that $\deg C$ is odd is easy to overlook. However, none of our main results holds if $\deg C$ is even; see 
Section~\ref{sec:evendegree} for counterexamples.
\end{remark}
The next result is stated in the language of the $V_0$ invariant of Rasmussen. We recall its definition in Section~\ref{sec:V0def} below.

\newtheorem*{thm:withconjecture}{Theorem \ref{thm:withconjecture}}
\begin{thm:withconjecture}
Suppose $C$ is a rational cuspidal curve of odd degree with two singular points $z_1$ and $z_2$. Let $K_1$ and $K_2$
be the corresponding links of the singular points. Then $V_0(K_1\# K_2)=V_0(K_1)+V_0(K_2)$.
\end{thm:withconjecture}

In general, if $K_1$ and $K_2$ are links of cuspidal singularities, then $V_0(K_1\# K_2)\le V_0(K_1)+V_0(K_2)$. An algorithm
for calculating $V_0$ from the semigroup of a singular point is given in Lemma~\ref{lem:v0isg} below; a more general statement
is given in Proposition~\ref{prop:v0isgtwo}. Here we give one important instance.

In Section \ref{sec:Lspacesknotsandtheirsums} (see Definition~\ref{def:oddevenLspace}) we introduce a simple but useful notion of an 
\emph{odd L--space knot} (based on the number of `stairs' in
the staircase complex). An algebraic knot $K$ is odd if and only if 
the $\delta$--invariant (or the three--genus) does not belong to the semigroup of the corresponding singular
point; see Proposition~\ref{prop:whenisodd}. We have the following result.

%\begin{theorem*}[see Theorem~\ref{thm:twoodds}]
\newtheorem*{thm:twoodds}{Theorem \ref{thm:twoodds}}
\begin{thm:twoodds}
Suppose $K_1$ and $K_2$ are odd L--space knots. Then $V_0(K_1\# K_2)<V_0(K_1)+V_0(K_2)$. 
\end{thm:twoodds}
In particular, combined with Theorem~\ref{thm:withconjecture}, we obtain the following obstruction to the existence of a rational cuspidal curve of odd degree having two singular points with odd links.

%\begin{theorem}\label{thm:sumoftwoodds}
\newtheorem*{thm:sumoftwoodds}{Theorem \ref{thm:sumoftwoodds}}
\begin{thm:sumoftwoodds}
Let $C$ be a rational cuspidal curve of odd degree with two singular points $z_1$ and $z_2$. Let $K_1$ and $K_2$ be links
of singularities of $z_1$ and $z_2$. Then at least one of $K_1$ and $K_2$ is an even L--space knot.
\end{thm:sumoftwoodds}

We illustrate the above application by a simple example. It is a well-known result (see \cite[Section 6.1.3]{Moe08})
but we give the first topological proof.
\begin{example}\label{ex:deg5}
A rational cuspidal curve of degree 5 cannot have two singular points with Puiseux sequences $(2;11)$ and $(2;3)$. It also cannot have
two singular points with Puiseux sequences $(2;7)$ and $(2;7)$.
\end{example}
In \cite[Section 6.1.3]{Moe08} Moe shows % it used to be shoes, I was tempted to leave it like that...
that a rational cuspidal curve with Puiseux sequences $(2;9)$ and $(2;5)$ actually exists.

This degree 5 example is quite remarkable from the following point of view. In \cite{BodN} Bodn\'ar and N\'emethi noticed that the
criterion of \cite{BL} does not actually restrict singular points, but only so-called multiplicity sequences. We do not give all the details,
but point out that the criterion of \cite{BL} is unable to distinguish the case of singular points $(2;11),(2;3)$ and $(2;9),(2;5)$.
We give more examples in Section~\ref{sec:moreexamples}.

\subsection{Outline of the proof}

The main idea of the proof comes from \cite{BL}, although technical problems already appear at an early stage. 
Consider a rational cuspidal
curve $C\subset\CP^2$ and let $N$ be a tubular neighborhood of $C$. Let $M=\partial N$ and set $W=\CP^2\setminus N$. As in \cite{BL}
we identify $M$ with a surgery on the sum of links of singularities of $C$. Moreover $H_k(W;\Q)=0$ for $k>0$. The latter fact
implies by \cite{OzSz03} that for any \spinc{} structure $\sss$ on $M$ that extends to $W$ we have $d(M,\sss)=0$, where $d$ is the 
Ozsv\'ath--Szab\'o $d$--invariant. The equality $d(M,\sss)=0$ was exploited in \cite{BL}.

In the present article we rely on a result of Hendricks and Manolescu, that for any \emph{Spin structure} $\sss$ on $M$ that
extends over $W$ we have $\ovd(M,\sss)=\und(M,\sss)=0$, where $\ovd$ and $\und$ are the invariants defined in \cite{HM}; see Section \ref{sec:ovd} below. We look
at the canonical Spin structure on $M$, that is, the one corresponding to $m=0$; see  Section~\ref{sec:undandovd} for notation.
The problem is that the canonical Spin structure extends over $W$ if and only if $\deg C$ is odd. Therefore our results are restricted
to curves of odd degree; see Section~\ref{sec:rationalcompl}.

If $C$ has one singular point, then $M$ is an L--space and it follows from \cite[Section 4.4]{HM}
that $\ovd(M,\sss)=\und(M,\sss)=d(M,\sss)$. 
In particular, our result says nothing new for rational cuspidal curves with one singular point.
However, if $C$ has more than one singular point, the condition  $\und(M,\sss)=\ovd(M,\sss)$ becomes restrictive.
The second, and actually, more difficult, 
part of the paper translates the equality $\und(M,\sss)=\ovd(M,\sss)$ into a tractable condition
on semigroups of singular points of $C$.
 
Throughout, we let $\F = \Z/2\Z$.

\begin{ack}
The authors would like to thank J\'ozsi Bodn\'ar, Kristen Hendricks, Adam Levine, Ciprian Manolescu, Andr\'as N\'emethi and Ian Zemke for fruitful discussions and comments.
\end{ack}
%
%
%
%
%
%
% ---------------------------------------------
\newpage
\section{Involutive Floer homology}\label{sec:hfi}
%
%
%
%
%
%
% ---------------------------------------------
\subsection{$\und$ and $\ovd$ invariants}\label{sec:ovd}
Let $Y$ be a rational homology sphere and $\sss$ a Spin structure on $Y$. In \cite{HM} Hendricks and Manolescu defined $\und(Y,\sss)$ and $\ovd(Y,\sss)$, which are refinements of the $d$-invariants of Ozsv\'ath and Szab\'o.

Let us quickly recall their construction. We assume that the reader is familiar with Heegaard Floer homology. Suppose $Y$ is a rational homology 3-sphere and $\sss$ is a Spin structure on $Y$.
There is a map $\iota\colon \cf^+(Y,\sss)\to \cf^+(Y,\sss)$, which induces an isomorphism on homology. The square of $\iota$ is  chain homotopic to the identity. 
 Let $Q$ be a formal variable of degree $-1$ with $Q^2=0$. Consider the map
\[ Q(1+\iota) \colon \cf^+ \rightarrow Q \cdot \cf^+[-1], \]
where the brackets denote shifts in grading, i.e. $C[n]_k=C_{k+n}$.
We will be interested in $\cfi^+$, the cone of $Q(1+\iota)$ together with an overall grading shift.
More precisely, let $\cfi^+$ denote the complex with underlying space $\cf^+[-1] \oplus Q \cdot \cf^+[-1]$ and differential
\[ \di= \begin{pmatrix} \partial  & 0 \\ Q(1+\iota) & \partial  \end{pmatrix}. \]

%Take two copies of $CF^+(Y,\sss)$ and denote the second one by $Q\cdot CF^+(Y,\sss)$, so that we can regard $Q$ as a map taking the first copy of $CF^+$ to the second copy.  Construct a map $Q(1+\iota)\colon CF^+\to Q\cdot CF^+[-1]$, where $[-1]$ means that the homological grading has been shifted down by $1$. Let $CFI^+$ be the cone of the map $Q(1+\iota)$, that is, a complex with underlying space $CF^+\oplus Q\cdot CF^+$ and the differential 
%\[ \di= \begin{pmatrix} \partial  & 0 \\ Q(1+\iota) & \partial  \end{pmatrix}. \]

It can be easily shown that the homology $\hfi^+$ splits (non-canonically) as a sum of two towers, $\mathcal{T}^+$ and $Q\mathcal{T}^+$, and the reduced part, which is finitely generated as an $\F$-module.  We will refer to $\mathcal{T}^+$ as the first tower and $Q \mathcal{T}^+$ as the second.
Here $\mathcal{T}^+\cong \F[U,U^{-1}]/U \F[U]$. 
We have the following definition.
\begin{definition}[see \expandafter{\cite[Section 5.1]{HM}}]
The lower and upper involutive correction terms $\und$ and $\ovd$ are given by
\begin{align*}
\und(Y,\sss)&=\min \left\{ r: \exists x\in \hfi^+_r(Y,\sss), x\in \iim(U^n), x\not\in \iim(U^nQ) \textrm{ for }n\gg 0\right\}-1\\
\ovd(Y,\sss)&=\min \left\{ r: \exists y\in \hfi^+_r(Y,\sss), y \neq 0, y\in \iim(U^nQ) \textrm{ for }n\gg0\right\}.
\end{align*}
\end{definition}
From the definition one obtains that $\ovd(Y,\sss)\ge d(Y,\sss)\ge \und(Y,\sss)$ and all three invariants differ by an even integer.
We have the following fundamental property of $\ovd$ and $\und$. The formulation is tailored for the applications in the present article.
\begin{theorem}[\expandafter{\cite[Proposition 5.4]{HM}}]\label{thm:hm-main}
Suppose $(Y,\sss)$ is a rational homology sphere with a Spin structure $\sss$. Assume that $Y=\partial W$, where $W$ is a smooth rational homology ball.
If $\sss$ extends to a Spin structure over $W$, 
then $\ovd(Y,\sss)=\und(Y,\sss)=0$.
\end{theorem}

%
%
%
%
%
%
% ---------------------------------------------
\subsection{$\und$ and $\ovd$ for large surgeries on knots}\label{sec:undandovd}
Knot Floer homology assigns to a knot in $S^3$ doubly filtered chain complexes $\cfk^\circ(K)$ for $\circ\in\{+,-,\infty\}$. (We
do not discuss the hat version in this paper.)
The filtration levels of an element $x\in \cfk^\circ$ are denoted by $\alpha(x)$ and $\beta(x)$ respectively. 

%We assume that the complex $\cfk^\infty$ is \emph{reduced}, that is, the differential of every bifiltered element $x$
%has the property that $\alpha(\partial x)+\beta(\partial x)<\alpha(x)+\beta(x)$; in other words, the differential decreases
%at least one filtration level. Refer to \cite{Krc} for more details.

There is a $U$ action on $\cfk^\circ$, which decreases the $\alpha$-- and $\beta$-- filtration levels by $1$ and decreases the homological grading by $2$.
The chain complex $\cfk^\infty$ can be used to calculate the Heegaard Floer homology of surgeries on $K$; see \cite{OzSz04, OSinteger}. Consider
the surgery $S^3_p(K)$ with $p>0$, $p \in \Z$. This manifold has an enumeration of \spinc{} structures by
integers $m\in[-p/2,p/2)$. Denote the corresponding \spinc{} structure by $\sss_m$. 
The \spinc{} structure $\sss_0$ is actually a Spin structure. If $p$ is even, then $\sss_{-p/2}$
is also a Spin structure, but we will focus on the Spin structure $\sss_0$.

Suppose $p\geq 2g(K)-1$  and consider the quotient complex
\[A^+_m:=\cfk^\infty(K)/\cfk^\infty(\alpha<0,\beta<m),\] 
 where $\cfk^\infty(\alpha<0,\beta<m)$ denotes elements whose first filtration level is less than $0$
\emph{and} whose second filtration level is less than $m$. By \cite[Theorem 4.4]{OzSz04} (cf. \cite[Theorem 1.1]{OSinteger} together with \cite[Theorem 1.2]{OSgenus}), $A^+_m$ is, up to an overall grading shift,  chain homotopy equivalent to the complex 
$\cf^+(S^3_p(K),\sss_m)$.
\begin{remark}
If $m=0$, it is enough to assume that $p\ge g(K)$ instead of $p\ge 2g(K)-1$.
\end{remark}

In \cite{HM}, Hendricks and Manolescu defined
a map $\iota\colon\cfk^\infty(K)\to\cfk^\infty(K)$, 
whose square is chain homotopic to the Sarkar map $\varsigma$ \cite{Sar,Zem1}. The map $\iota$ preserves the homological grading, but is skew-filtered;
that is, $\alpha(\iota(x))\le\beta(x)$ and $\beta(\iota(x))\le \alpha(x)$.
In particular, $\iota$ descends to a map $\iota\colon A^+_0\to A^+_0$.
We have the following compatibility relation between the map $\iota$ on $A^+_0$ and $\iota$ on $\cf^+(Y)$ defined above.

\begin{proposition}[see {\cite[Equation (24)]{HM}}]\label{prop:iotaonknots}
Suppose that $K$ is a knot in $S^3$ and $p\ge g(K)$, $p \in \Z$. Consider $Y=S^3_p(K)$ endowed with Spin structure $\sss_0$
and identify the chain complex $\cf^+(Y,\sss_0)$ with the complex $A_0^+$ as above. Then the action of $\iota$
on $A^+_0$ induces the action of $\iota$ on $\cf^+(Y,\sss)$. In particular, the homology of the cone complex 
\[AI^+:=A^+_0\oplus Q\cdot A^+_0\]
with differential $\di=\begin{pmatrix} \partial  & 0 \\ Q(1+\iota) & \partial  \end{pmatrix}$ 
is isomorphic (up to an overall grading shift) to the homology of $\cfi^+(Y,\sss_0)$.
\end{proposition}

\subsection{The $V_0$, $\ovV_0$ and $\unV_0$ invariants}\label{sec:V0def}
Recall that for $m\in\Z$
the invariant $V_m(K)$ was defined by the property that $-2V_m(K)$ is the minimal grading of a (non-zero) element in $H_*(A^+_m)$
that is in the image of $U^n$ for all positive $n$. The invariants $V_m$ were first defined by Rasmussen in \cite{Ras}. (He
uses a slightly different invariant $h_i$ with essentially the same meaning.)
The notation we use is that of \cite{NiWu} and we focus on the case $m=0$, that is, on the invariant $V_0(K)$.
It follows from \cite[Proposition 1.6]{NiWu} that if $p \geq g(K)$  
then the $d$-invariant of $p$-surgery on $K$
satisfies 
\begin{equation}\label{eq:v0surg}
d(S^3_p(K),\sss_0)=\frac{p-1}{4}-2V_0(K).
\end{equation}

There is another description of $V_0(K)$.
Consider the set of graded
elements $x_1,\ldots,x_m\in \cfk^\infty(K)$ such that each $x_i$ is a generator of $\hfk^\infty(K)$ at grading $0$.
The invariant $V_0(K)$ is equal to
\begin{equation}\label{eq:v0}
V_0(K)=\min_{j=1,\ldots,m} \max(\alpha(x_j),\beta(x_j)).
\end{equation}
To see that the two definitions are equivalent, notice that an element in $H_*(A^+_0)$ of minimal grading such that it is in the image of $U^n$ for all positive $n$ must be of form $U^kx_j$ for some $k,j$.
On the one hand, 
$U^kx_j=0\in H_*(A^+_0)$ if $k>\max(\alpha(x_j),\beta(x_j))$ and as all the $U^kx_i$ for $i=1,\ldots,m$ are homologous in $A^+_0$, we deduce that
there can be no homologically non-trivial element in $A^+_0$ of the form $U^k x_j$ if $k>\min_j\max(\alpha(x_j),\beta(x_j))$. This shows the `$\le$' 
part of \eqref{eq:v0}. On the other hand, by definition there exists an element in $A_0^+$ at grading $-2V_0(K)$ that is homologically non-trivial
and that is in the image of $U^n$ for any $n$. In particular such an element must be of form $U^kx_j$ for some $k$ and $j$. Looking at
the gradings implies that $k=V_0(K)$ and as $U^kx_j\neq 0\in H_*(A_0^+)$ we infer that $V_0(K) \ge \max(\alpha(x_j),\beta(x_j))$.

Given Proposition~\ref{prop:iotaonknots} one introduces invariants $\ovV_0$ and $\unV_0$ defined as follows. Consider
the complex $AI^+$ as above. Define
\begin{equation}\label{eq:ovddef}
\begin{split}
\unV_0&=\max \{ r : \exists x\in H_{-2r}(AI^+),x\in \iim U^n, x\notin \iim(U^nQ)\textrm{ for }n\gg 0\}.\\
\ovV_0&=\max \{ r : \exists y\in H_{-2r-1}(AI^+), y \neq 0, y\in \iim (U^nQ) \textrm{ for }n\gg 0\}.
\end{split}
\end{equation}
We have the following result.
\begin{theorem}[see \expandafter{\cite[Theorem 1.6]{HM}}]
Suppose $p\ge g(K)$. Then $\ovd$ and $\und$ of surgeries on a knot $K$ are related to $\ovV_0$ and $\unV_0$ by the following
formula:
\begin{equation}\label{eq:ovd}
\begin{split}
\ovd(S^3_p(K),\sss_0)&=\frac{p-1}{4}-2\ovV_0(K)\\
\und(S^3_p(K),\sss_0)&=\frac{p-1}{4}-2\unV_0(K).
\end{split}
\end{equation}
\end{theorem}
\begin{remark}
The original definition of $\unV_0$ and $\ovV_0$ in \cite{HM} is via \eqref{eq:ovd}. However in \cite[Section 6.7]{HM}, the
invariants $\ovV_0$ and $\unV_0$ are determined by calculating minimal gradings of elements generating
the two towers. Equation~\eqref{eq:ovddef} above is a reformulation of their definition.
\end{remark}

By \cite[Proposition 5.1]{HM}, we have
\begin{equation}\label{eq:ovvunv}
\ovV_0(K)\le V_0(K)\le\unV_0(K).
\end{equation}

%
%
%
%
%
%
% ---------------------------------------------
\subsection{Involutively simple knots}
We focus on knots for which $\unV_0(K)=\ovV_0(K)$. Since this condition will be used quite frequently, we give it a name.
\begin{definition}
A knot $K\subset S^3$ is called \emph{involutively simple} if $\unV_0(K)=\ovV_0(K)$.
\end{definition}
\begin{example}
As was shown in \cite[Section 7]{HM}, L--space knots are involutively simple; however mirrors of non-trivial L--space knots
are not. (See Section \ref{sec:Lspaceknots} for the definition of L--space knots.)
\end{example}

Let $K$ be a knot in $S^3$ and denote by $x_1,\ldots,x_n$ all possible elements of $\cfk^\infty(K)$ at grading $0$ which generate
$\hfk^\infty(K)$.
\begin{proposition}\label{prop:C1C2}

The invariant $\unV_0(K)$ can be equal to $V_0(K)$ in precisely two cases:
\begin{itemize}
\item[(C1)] There exists a generator $x_i$ at grading $0$  which is fixed by $\iota$ and minimizes $\max(\alpha(x_j),\beta(x_j))$ in \eqref{eq:v0}.
\item[(C2)] There exists a generator $x_i$ at grading $0$ minimizing $\max(\alpha(x_j),\beta(x_j))$ such that there exists $y_i$
with $\partial y_i=(1+\iota)x_i$ and $\max(\alpha(x_i),\beta(x_i)) \ge \max(\alpha(y_i),\beta(y_i))$.
\end{itemize}
\end{proposition}

\begin{remark}
Since $\iota$ is skew-filtered, the requirement in (C1) that $x_i$ is fixed by $\iota$ implies that $\alpha(x_i)=\beta(x_i)$.
\end{remark}

\begin{proof}[Proof of Proposition \ref{prop:C1C2}]
Recall that $AI^+$ is the cone of $Q(1+\iota)$. As $H_r(A^+_0)=\F$ for large even $r$, the homology exact triangle
for the cone implies that
generators of the first tower are of the form $U^{\delta_i}(x_i+Qy_i)$ for some $i=1,\ldots,n$, some integer  $\delta_i$
and a graded element $y_i\in\cfk^\infty(K)$. The element $y_i$ is only determined up to adding a boundary. We (partially)
fix this indeterminacy by chosing a representative $y_i$ in such a way that $\max(\alpha(x_i+Qy_i)),\beta(x_i+Qy_i))$ is minimal. More specific determination of $y_i$ will not be needed in the proof.

The fact that $x_i+Qy_i$ is a cycle translates into
\[ \di(x_i+Qy_i)=Q(1+\iota)x_i+Q\partial y_i=0,\] 
so $(1+\iota) x_i=\partial y_i$.
 
A consequence of \eqref{eq:ovddef} is that $\unV_0=\min_i\max(\alpha(x_i+Qy_i),\beta(x_i+Qy_i))$.  Combined with \eqref{eq:v0},
the assumption that $\unV_0=V_0$ implies that
\begin{equation}\label{eq:unv0v0}
\min_i\max(\alpha(x_i+Qy_i),\beta(x_i+Qy_i))=\min_i\max(\alpha(x_i),\beta(x_i)).
\end{equation}
Clearly $\alpha(x_i+Qy_i)\ge\alpha(x_i)$ and $\beta(x_i+Qy_i)\ge\beta(x_i)$. Therefore there exists an index $i$ which simultaneously  minimizes both the left and right hand sides.
 For this
index we have 
\begin{itemize}
	\item $(1+\iota) x_i=\partial y_i$;
	\item $x_i+Qy_i$ minimizes the left hand side of \eqref{eq:unv0v0}. 
\end{itemize}
It follows that
\[\max(\alpha(Qy_i),\beta(Qy_i))\le \max(\alpha(x_i),\beta(x_i)).\]

By the definition, $Q$ preserves the filtration levels, hence from the above equation we obtain
$\max(\alpha(y_i),\beta(y_i))\le \max(\alpha(x_i),\beta(x_i))$. If $y_i\neq 0$,
we are in case (C2). If $y_i=0$, we have $(1+\iota)x_i=0$; that is, $x_i=\iota x_i$, which is case (C1).
\end{proof}
\begin{remark}
The case (C1) is in fact a special subcase of (C2). However in the applications later on, it will be convenient to distinguish
between the two.
\end{remark}

A consequence of Proposition~\ref{prop:C1C2} is that sometimes we are able to show that a knot is not involutively simple 
merely by looking at the filtration levels of the generators of $CFK^\infty$.
A detailed discussion of the case of connected sums of L--space knots is given in the next section.

%
%
%
%
%
%
% ---------------------------------------------

\section{L--space knots and their sums}\label{sec:Lspacesknotsandtheirsums}
\subsection{Singular points, links and semigroups}\label{sec:semigroups}

Let $z\in C$ be a singular point of an algebraic curve. We consider curves in $\CP^2$, but as our analysis in this section is local, we
can assume that $z\in\C^2$ and $C$ is a plane algebraic curve. The \emph{link of singularity} is defined as $L=S^3_z\cap C\subset S^3_z$,
where $S^3_z$ is a small sphere around $z$. The singularity is called \emph{cuspidal} if $L$ has one component. This is equivalent to
saying that the intersection of $C$ with a small ball around $z$ is homeomorphic to a disk.

With each cuspidal singular point we can associate a sequence of positive integers $(p;q_1,\ldots,q_n)$, which is called a Puiseux sequence. 
A singular point with Puiseux sequence $(p;q_1,\ldots,q_n)$ is topologically equivalent to a singular point 
parametrized locally by $\C\to\C^2$, $t\mapsto (t^p,t^{q_1}+\ldots+t^{q_n})$, and topological equivalence means that the correspondig
links of singular points are isotopic. We refer to \cite{EN,Wa}
for more details.

To a singular point $z$ we can associate a numerical semigroup $S(z)$; see \cite[Chapter 4]{Wa}. The semigroup is the set of all non-negative
numbers that can be realized as local intersection indices of $C$ and some other complex curve $D$ not containing $C$. For example, for a singularity
with a Puiseux sequence $(p;q)$, the semigroup
is generated by $p$ and $q$. By convention, zero is always an element of the semigroup.

If $S(z)$ is the semigroup of a singular point, the \emph{gap set} $G=\Z_{\ge 0}\setminus S(z)$ is a finite set. Its cardinality is the 
$\delta$--invariant of $z$, which is equal to half the Milnor number and also to the three--genus of the link of singularity. Moreover
the expression
\[1+(t-1)\sum_{j\in G}t^j\]
is equal to the Alexander polynomial of the link of the singularity; see \cite[Chapter 4]{Wa}.

\subsection{L--space knots}\label{sec:Lspaceknots}
Recall that a rational homology sphere $Y$ is an \emph{L--space} if $\operatorname{rk} \widehat{\hf}(Y) = |H_1(Y; \Z)|$ and that a knot $K \subset S^3$ is an \emph{L--space knot} if it admits a positive L--space surgery; see \cite{os:lspace}. 
For us the main source of examples comes from the following result of Hedden \cite{Hed}.
\begin{theorem}[\cite{Hed}]
  A link of a cuspidal singularity is an L--space knot.
\end{theorem}

The $\cfk^\infty$ complex for an L--space knot has a particularly simple form. It is usually described in terms of staircase complexes.
A \emph{staircase complex} is a bifiltered complex with generators $x_0,\ldots,x_n$, $y_0,\ldots, y_{n-1}$ and boundary operator $\partial y_j=x_j+x_{j+1}$. The filtration levels, still denoted by $\alpha$ and $\beta$, are such that $\beta(x_i)=\beta(y_i)>\beta(x_{i+1})$ and 
$\alpha(x_i)<\alpha(y_i)=\alpha(x_{i+1})$. Therefore, the filtration levels of $y_0,\ldots,y_{n-1}$ are determined by
the filtration levels of $x_0,\ldots,x_n$. The grading of the generators is such that the $x_j$ have grading $0$ and the $y_j$
have grading $1$. We denote the staircase complex by $St(x_0,x_1,\ldots,x_n)$. %, or, if no risk of confusion arises, by $St$.

For an L--space knot $K$, there is an associated staircase $St(K)$ with the property that $St(K)\otimes\F[U,U^{-1}]$
is bifiltered chain homotopy equivalent to $\cfk^\infty(K)$.
The complex $St(K)$ can be calculated from the Alexander polynomial $\Delta_K$. There are various accounts for this fact; see
for instance \cite{Pet}. We present a short description of the procedure using the language of semigroups.

Write the Alexander polynomial of $K$ as 
\begin{equation}\label{eq:deltak}
\Delta_K(t)=t^{a_0}-t^{a_1}+t^{a_2}-\ldots+t^{a_{2m}},
\end{equation}  where $0=a_0<a_1<\ldots<a_{2m}=2g(K)$. 
The last formula can be rewritten in the following way:
\[\Delta_K(t)=1+(t-1)(t^{g_1}+\ldots+t^{g_r}),\]
for some integers $0<g_1<\ldots<g_r$. Define the \emph{gap set} of $K$ to be $G=\{g_1,\ldots,g_r\}$. Let 
\begin{equation}\label{eq:defS}
S(K)=\{x\in\Z, x\ge 0,x\notin G\}.
\end{equation}
If $K$ is an algebraic knot, then in Section~\ref{sec:semigroups} we saw that $S(K)$ is the semigroup of the corresponding 
singular point. For general L--space
knots, $S(K)$ will not necessarily have a semigroup structure; see, for example, \cite[Example 2.3]{BCG}.

We pass to the construction of the staircase complex. For an exponent $a_j$ of the Alexander polynomial place a generator at bifiltration level
\[(\#S(K)\cap[0,a_j),\#(\Z\setminus S(K))\cap[a_j,\infty)).\] % it is unreadable if the formula is broken over two lines. 
Call this element $x_k$ if $j=2k$ and $y_{k}$ if $j=2k+1$ for some $k\in\Z$.
Notice that the elements $a_j$ have the following property: if $j$ is odd, then $a_j\notin S(K)$, but $a_j-1\in S(K)$; conversely, if $j$ is
even, then $a_j\in S(K)$, but $a_j-1\notin S(K)$. It follows that
$\beta(x_i)=\beta(y_i)$ and $\alpha(y_i)=\alpha(x_{i+1})$. The staircase $St(x_0,\ldots,x_n)$ constructed
in this way is the staircase $St$ of the L--space knot $K$. Throughout the paper we refer to $x_0,x_1,\ldots,x_n$
as the $x$--type generators and $y_0,\ldots,y_{n-1}$ as the $y$--type generators of the staircase $St$.

\begin{example}\label{ex:staircase56}
Consider the singularity of a complex curve in $\C^2$ given $x^5-y^6=0$, which has Puiseux sequence $(5;6)$. The link of
the singular point is the torus knot $T(5,6)$ and the semigroup is generated by $5$ and $6$. The Alexander polynomial is
$1-t+t^5-t^7+t^{10}-t^{13}+t^{15}-t^{19}+t^{20}$, so $a_0=0$, $a_1=1$, $a_2=5$, $a_3=7$, $a_4=10$, $a_5=13$, $a_6=15$, $a_7=19$ and $a_8=20$. The bifiltration
levels of the generators are, respectively $(0,10)$, $(1,10)$, $(1,6)$, $(3,6)$,
$(3,3)$, $(6,3)$, $(6,1)$, $(10,1)$ and $(10,0)$. The generators
at bifiltration level $(0,10)$, $(1,6)$, $(3,3)$, $(6,1)$, $(10,0)$ are
$x$--type generators, while the others are $y$--type generators. 
See Figure~\ref{fig:T56}.
\end{example}
\begin{figure}
\begin{tikzpicture}[scale=0.5]
\draw[dotted,color=black!60](-1,-1) grid (10,10);
\draw[->] (0,-1)--(0,11);
\draw[->] (-1,0)--(11,0);
\draw[fill=black] (10,0) circle (0.1);
\draw[fill=black,thick] (10,0) -- (10,1) circle (0.1);
\draw[fill=black,thick] (10,1) -- (06,1) circle (0.1);
\draw[fill=black,thick] (06,1) -- (06,3) circle (0.1);
%\draw[fill=black,thick] (06,3) -- (03,3) circle (0.1);
%\draw[fill=black,thick] (03,3) -- (03,6) circle (0.1);
\draw[fill=black,thick] (03,6) -- (01,6) circle (0.1);
\draw[fill=black,thick] (01,6) -- (1,10) circle (0.1);
\draw[fill=black,thick] (1,10) -- (0,10) circle (0.1);
%\draw[green!50!black,ultra thick] (6,3) -- (3,3) -- (3,6);
\draw[white!65!black,ultra thick] (6,3) -- (3,3) -- (3,6);
\fill[black] (6,3) circle (0.1);
\fill[black] (3,3) circle (0.1);
\fill[black] (3,6) circle (0.1);
\end{tikzpicture}
\caption{The staircase complex for the torus knot $T(5,6)$. The length of the grey middle segments is the stretch, which in this case is $3$.}\label{fig:T56}
\end{figure}
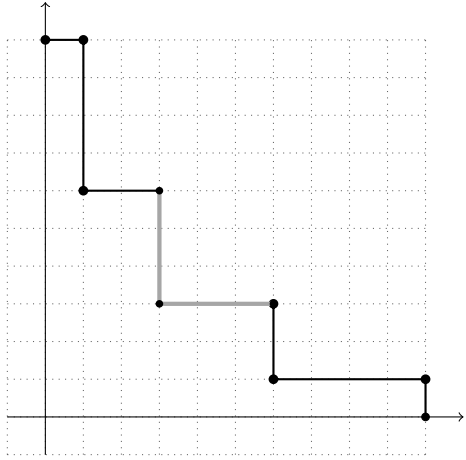
There is a geometric way of describing the staircase as a broken line made of segments joining
$(\#S(K)\cap[0,m),\#(\Z\setminus S(K))\cap[m,\infty))$ to $(\#S(K)\cap[0,m+1),\#(\Z\setminus S(K))\cap[m+1,\infty))$ for $m\in\Z$. The corners of
the line correspond to generators of the staircase; more precisely, the `L'--shaped corners correspond to $x$--type generators
of the staircase and the `7'--shaped corners correspond to $y$--type generators of the staircase. The horizontal and vertical lines indicate
the differentials in the chain complex.

We will also need the following standard fact:
\begin{lemma}\label{lem:v0isg}
The $V_0$ invariant of an L--space knot $K$ is equal to $\# S(K)\cap[0,g)$, where $g$ is the three--genus of $K$.
\end{lemma}
\begin{proof}[Sketch of proof]
The result is well-known, although the formulation might not be very common. We follow \cite{BL}. The invariant $V_0(K)$ can be identified
with the value of $J(0)$ in \cite[Section 4.3]{BL}. Using \cite[Lemma 6.2]{BL} we conclude that $J(0)=\#S(K)\cap[0,g)$.

There is also a direct proof, which uses \eqref{eq:v0} and the construction of the staircase. We leave it as an exercise.
\end{proof}

\subsection{Stretch of L--space knots}\label{sec:stretchLspace}
We will need two more definitions.
\begin{definition}\label{def:oddevenLspace}\
\begin{itemize}
\item The staircase $St(x_0,\ldots,x_n)$ is called \emph{even} or \emph{odd} according to whether $n$ is an even or an odd number.
By a slight abuse of language, an L--space knot will be called \emph{even} or \emph{odd} whenever its staircase is even or odd.
\item The \emph{stretch} of a staircase is the difference $\alpha(x_{k+1})-\alpha(x_k)$, where $k=n/2$ if $n$ is even or $k=(n-1)/2$
if $n$ is odd.
\end{itemize}
\end{definition}
\begin{example}
		The knot $T(5,6)$ discussed in Example~\ref{ex:staircase56} is even. The stretch is equal to $3$, and $V_0(T(5,6))$ is also $3$. See Figure \ref{fig:T56}.
\end{example}
We have the following result.
\begin{proposition}\label{prop:whenisodd}
Suppose $K$ is an L--space of genus~$g$. Then
$K$ is even if and only if $g\in S(K)$ if and only if $m$ is even (where as in \eqref{eq:deltak} above $m$ is such that $2m+1$ is the number of terms in the Alexander polynomial). 
Furthermore, the stretch is the maximal integer $w$ such that
$g,g+1,\ldots,g+w-1$  either all belong to $S(K)$ or none of them does. In other words, the stretch is given by $a_{m+1}-a_m$.
\end{proposition}
\begin{proof}
The Alexander polynomial $\Delta_K$ is given by $t^{a_0}-t^{a_1}+\ldots+t^{a_{2m}}$ with $a_0=0$ and $a_{2m}=2g$. By the symmetry
of the Alexander polynomial we have $a_i=2g-a_{2m-i}$. It follows that $a_m=g$. Now if $m$ is even, then $a_m\in S(K)$ and
$a_m$ corresponds to the $x_{m/2}$ vertex of the staircase. If $m$ is odd, then $a_m\notin S(K)$ and $a_m$ corresponds to
the $y_{(m-1)/2}$ vertex of the staircase. This proves the first part.

To prove the second part we assume that $m$ is even; the case $m$ odd is similar. Then $a_{m}\in S(K)$ and the
first element not belonging to $S(K)$ that is greater than $a_m$ is $a_{m+1}$. As $m$ is even, $a_m$ corresponds to the $x_k$ vertex
for $k=m/2$ and $a_{m+1}$ corresponds to the vertex $y_k$.
By the definition of the $\alpha$--gradings, we have $\alpha(x_{k+1})=\alpha(y_k)$ and 
$\alpha(y_k)-\alpha(x_k)=\#S(K)\cap[0,a_{m+1})-\#S(K)\cap[0,a_m)=a_{m+1}-a_m$.
\end{proof}

There is a useful perspective for understanding even and odd staircases. Namely, by symmetry, each staircase has exactly one vertex $z$ such that 
$\alpha(z)=\beta(z)$. This vertex is an $x$--vertex if the staircase is even and an $y$--vertex if the staircase is odd. It is not hard to see,
e.g., via \eqref{eq:v0},
that for this $z$ we have $\alpha(z)=\beta(z)=V_0(K)$. With this in mind, we can state the following rather simple corollary
of Proposition~\ref{prop:whenisodd}, which will be used extensively
 in the remaining part of the paper.
\begin{corollary}\label{cor:easy}
Suppose $K$ is an L--space knot and $St(x_0,\ldots,x_n)$ is its staircase. Suppose $s$ is the stretch of $K$. If $n$ is odd, then
$x_{(n-1)/2)}$ is a generator at bifiltration level $(V_0(K)-s,V_0(K))$, $x_{(n+1)/2}$ is a generator at bifiltration level $(V_0(K),V_0(K)-s)$
and $y_{(n-1)/2}$ is a generator at bifiltration level $(V_0(K),V_0(K))$.

If $n$ is even, then $y_{n/2}$ is a generator at bifiltration level $(V_0(K),V_0(K))$.
\end{corollary}

We have the following number theoretic criterion for the stretch of torus knots.
\begin{theorem}\label{thm:sch}
Let $p,q>1$ be coprime integers with $q>p$. Let $\frac{q}{p}=[a_0,\ldots,a_k]$ be the regular continued fraction expansion with $a_k>1$.
Then the stretch of the torus knot $T(p,q)$ is equal to $\intfrac{a_k-1}{2}+1$.
\end{theorem}
\noindent The proof of Theorem~\ref{thm:sch} is contained in Appendix~\ref{sec:proof}. We now pass to the main result of this section.

\begin{theorem}\label{thm:oddstretch}
Suppose $K_1$ is an odd L--space knot with stretch $s$. Suppose also $K_2$ is an L--space knot, or a connected sum of L--space knots, with genus~$g$.
If $g<s$, then $K=K_1\#K_2$ is not involutively simple.
\end{theorem}

\noindent The proof of Theorem~\ref{thm:oddstretch} is given in Section~\ref{sec:proofoddstretch}.

\begin{example}
If $K_1$ is odd and has stretch greater than $1$ and $K_2$ is a trefoil, then $K_1\# K_2$ is not involutively simple. For example if $K_1$
is a $T(4,n)$ torus knot and $n$ is coprime with $4$, then the stretch is $2$ by Theorem~\ref{thm:sch}. If, additionally, $(n\bmod 8)>4$, then
$T(4,n)$ is odd and then $T(4,n)\# T(2,3)$ is not involutively simple.
\end{example}

\subsection{Sums of two L--space knots}
In order to study $\ovV_0$ and $\unV_0$ for sums of L--space knots we need to have a description of the map $\iota$ for sums
of L--space knots. Such a description follows from a recent result of Zemke \cite{Zem2} generalizing an analogous
statement for the action of $\iota$ on the Heegaard Floer chain complexes of connected sums of three--manifolds \cite{HMZ}.
We begin by recalling some notation from \cite[Section 3]{Zem1} and \cite{Sar}.
Suppose $K\subset S^3$ is a knot and $y\in\cfk^\infty(K)$ is a filtered element. Write $\partial y=x_1+\dots+x_k$, where $x_1,\ldots,x_k$
are also filtered elements. We define
\begin{align*}
\Phi(y)&=\sum_{i=1}^k (\alpha(y)-\alpha(x_i))x_i\\
\Psi(y)&=\sum_{i=1}^k (\beta(y)-\beta(x_i))x_i.
\end{align*}
The maps $\Phi,\Psi\colon\cfk^\infty(K)\to\cfk^\infty(K)$ are sometimes referred to as \emph{formal derivatives} of $\partial$. 
An trivial but important consequence of the definition is that if $y$ is a cycle, then $\Phi(y)=\Psi(y)=0$. The
definition of the map $\iota$ on the connected sum involves the maps $\Phi$ and $\Psi$ on $\cfk^\infty$ of the summands,
as is shown in the following result of Zemke.

\begin{theorem}[{\cite{Zem2}}]\label{thm:zemke0}
Suppose $K_1$ and $K_2$ are knots and $\iota_1,\iota_2$ are the maps $\iota$ for $\cfk^\infty(K_1)$ and $\cfk^\infty(K_2)$. Let
 $\Phi_{i}$ and $\Psi_{i}$ be the formal derivatives of the differential of $\cfk^\infty(K_{i})$, $i=1,2$ as described above.
Then with the identification $\cfk^\infty(K_1\#K_2)=\cfk^\infty(K_1)\otimes
\cfk^\infty(K_2)$, the map $\iota$ on $\cfk^\infty(K_1\#K_2)$ is filtered chain homotopy equivalent to the map
\[u\otimes v\mapsto \iota_1(u)\otimes\iota_2(v)+\Phi_1\iota_1(u)\otimes\Psi_2\iota_2(v).\]
\end{theorem}

As we see, in general, the map $\iota$ on the connected sum is not merely a tensor product of the $\iota$ maps on the summands. However,
if $K_1$ and $K_2$ are L--space knots and $u,v$ are $x$--type generators, then $\partial u=\partial v=0$ and so $\Phi_1(u)=\Psi_2(v)=0$.
Therefore from Theorem~\ref{thm:zemke0} we obtain the following result, which will be sufficient for our purposes in understanding the action of $\iota$ on the connected sum of two L--space knots.

\begin{theorem}\label{thm:zemke}
Suppose $K_1$ and $K_2$ are L--space knots with staircases generated respectively by $x_0,\ldots,x_n,y_0,\ldots,y_{n-1}$ and
$x_0',\ldots,x_m',y_0',\ldots,y_{m-1}'$. Then the action of $\iota$ on $\cfk^\infty(K_1\# K_2)$ has the property that
$\iota(x_i\otimes x_j')=x_{n-i}\otimes x_{m-j}$.
\end{theorem}

This result allows us to find a tractable obstruction to knots being involutively simple. 
\begin{theorem}\label{thm:sumoftwo}
Let $K_1$ and $K_2$ be L--space knots and set $K=K_1\# K_2$. If $V_0(K)<V_0(K_1)+V_0(K_2)$, then $K$ is not involutively simple.
\end{theorem}
Recall from \cite[Proposition 6.1]{BCG} that $V_0(K_1 \# K_2) \leq V_0(K_1) + V_0(K_2)$. The proof of Theorem~\ref{thm:sumoftwo} is postponed until Section~\ref{sec:proofoftwo}.

\begin{theorem}\label{thm:twoodds}
Let $K_1$ and $K_2$ be odd L--space knots. Then $V_0(K_1\# K_2) < V_0(K_1)+V_0(K_2)$.
\end{theorem}
\begin{proof}
Let $x_0,\ldots,x_n,y_0,\ldots,y_{n-1}$ be generators of the staircase of $K_1$ and let $x_0',\ldots,x_m'$, $y_0',\ldots,y_{m-1}'$ be
generators of the staircase of $K_2$. By assumption $n$ and $m$ are odd. By Corollary~\ref{cor:easy} we have that $x_{(n-1)/2}$ is at
bifiltration level $(V_0(K_1)-s_1,V_0(K_1))$ and $x'_{(m+1)/2}$ is at bifiltration level $(V_0(K_2),V_0(K_2)-s_2)$, where $s_1$ and
$s_2$ are the stretches of the respective staircases (see Definition~\ref{def:oddevenLspace}); here we only need that $s_1,s_2\ge 1$.

The product $x_{(n-1)/2}\otimes x'_{(m+1)/2}$ is at bifiltration level $(V_0(K_1)+V_0(K_2)-s_2,V_0(K_1)+V_0(K_2)-s_1)$ and it is a
generator of the homology of $\cfk^\infty(K_1\# K_2)$ at grading zero. By \eqref{eq:v0} this implies that 
$V_0(K_1\#K_2)\le V_0(K_1)+V_0(K_2)-\min(s_1,s_2)$.
\end{proof}
The following result, combined with Theorem \ref{thm:sumoftwoov} below, provides the converse to Theorem~\ref{thm:sumoftwo}.
\begin{theorem}\label{thm:unv0}
Let $K_1$ and $K_2$ be L--space knots and set $K=K_1\#K_2$. Suppose that $V_0(K)=V_0(K_1)+V_0(K_2)$. Then $\unV_0(K)=V_0(K)$.
\end{theorem}
\begin{proof}
By Theorem~\ref{thm:twoodds} at least one of the knots must be even. So suppose $K_1$ is even. By Corollary~\ref{cor:easy} there exists
an $x$--type generator of the staircase $St(K_1)$ at bifiltration level $(V_0(K_1),V_0(K_1))$. Call this element $x$.

If $K_2$ is also even, we take an $x$--type generator $x'$ of the staircase $St(K_2)$ at bifiltration level $(V_0(K_2),V_0(K_2))$.
The tensor product $x\otimes x'$ is at bifiltration level $(V_0(K),V_0(K))$. By Theorem~\ref{thm:zemke}
this element is fixed by $\iota$. Hence $x\otimes x'$ regarded as an element in $AI^+$ is a cycle that generates a tower.
Therefore $\unV_0(K)\le \max(\alpha(x\otimes x'),\beta(x\otimes x'))=V_0(K)$. As $\unV_0(K)\ge V_0(K)$ by \eqref{eq:ovvunv}, we have finished the proof if $K_2$
is even.

If $K_2$ is odd, by Corollary~\ref{cor:easy}
there exists a $y$--type generator $y'$ of the staircase $St(K_2)$ at bifiltration level $(V_0(K_2),V_0(K_2))$.
Let $\partial y'=x_1'+x_2'$, where $x_1',x_2'$ are $x$--type generators of $St(K_2)$. Without loss of generality, suppose that $\alpha(x_1)<\alpha(x_2)$. 
Take the element $x\otimes x_1'$. Then $\iota(x\otimes x_1')=x\otimes x_2'$ and hence $(1+\iota)x\otimes x_1'=\partial (x\otimes y')$. 
It follows that $x\otimes x_1'+Q(x\otimes y')$ generates the first tower of $AI^+$. The bifiltration level of that element is
easily calculated as $(V_0(K),V_0(K))$.
\end{proof}

We now give a numerical criterion for checking the hypothesis of Theorem~\ref{thm:sumoftwo}.
\begin{proposition}\label{prop:v0isgtwo}
Suppose $K_1$ and $K_2$ are L--space knots and $S_1=S(K_1)$, $S_2=S(K_2)$, as in \eqref{eq:defS}. (If $K_1,K_2$ are algebraic knots, then
$S_1$ and $S_2$ are the corresponding semigroups.) Let $g_1$ and $g_2$ be the genera of $K_1$ and $K_2$. Then
\[V_0(K_1\#K_2)=\min_{i+j=g_1+g_2}(\#S_1\cap[0,i)+\#S_2\cap[0,j)).\]
\end{proposition} 
\begin{proof}[Sketch of proof]
The result is a simple consequence of the K\"unneth formula for $\cfk^\infty$; we sketch the proof using
the notation of  \cite{BL}.
Using \cite[Proposition 5.1]{BL} we identify $V_0(K_1\# K_2)$ with the value of $J_{K_1\#K_2}(0)$, which is the same as
$I_{K_1\# K_2}(g_1+g_2)$. By \cite[Lemma 6.4]{BL}
we have $I_{K_1\# K_2}(g_1+g_2)=R_{K_1\# K_2}(g_1+g_2)$. The latter by definition is $\min_{i+j=g_1+g_2}(\#S_1\cap[0,i)+\#S_2\cap[0,j))$.
\end{proof}
Until now we have addressed the question of whether $\unV_0(K)=V_0(K)$, where $K$ is a connected sum of L--space knots. Another question is whether
$\ovV_0(K)=V_0(K)$. The answer is particularly simple for sums of two L--space knots: the invariant $\ovV_0(K)$ does not give any new information.

\begin{theorem}\label{thm:sumoftwoov}
Suppose $K$ is a connected sum of two L--space knots. Then $\ovV_0(K)=V_0(K)$.
\end{theorem}
In particular, if $K$ is a connected sum of two L--space knots, then it is involutively simple if and only if $\unV_0(K)=V_0(K)$. 
The proof of Theorem~\ref{thm:sumoftwoov} is deferred to Section~\ref{sec:proofovv}.

\begin{example}
Let $K_1=T(6,17)$, $K_2=T(4,11)$ and $K_3=T(4,13)$. Then $V_0(K_1)=12$, $V_0(K_2)=5$, $V_0(K_3)=7$, $V_0(K_1\# K_2)=17$ and
$V_0(K_1\# K_3)=18$. We see that $K_1\# K_3$ is not involutively simple. On the other hand $K_1\# K_2$ is involutively simple.
\end{example}

Combining Theorem~\ref{thm:sumoftwo} with Theorem~\ref{thm:unv0}, Proposition~\ref{prop:v0isgtwo} and Theorem~\ref{thm:sumoftwoov}
we obtain a tractable numerical  
criterion for the sum of two L--space knots to be involutively simple.
\begin{theorem}\label{thm:finaltwo}
The sum of two L--space knots $K_1$ and $K_2$ is involutively simple if and only if 
\[\min_{i+j=g_1+g_2}(\#S_1\cap[0,i)+\#S_2\cap[0,j))=\#S_1\cap[0,g_1)+\#S_2\cap[0,g_2).\]
\end{theorem} 

%----------------------------------------------------
\section{Rational cuspidal curves} 

\subsection{Rational cuspidal curves and their complements}\label{sec:rationalcompl}
Let $C\subset\CP^2$ be a rational cuspidal curve. %, that is, a reduced algebraic curve of genus $0$, all of whose critical
%points have precisely one branch. Put differently, $C$ is an algebraic curve in $\CP^2$ homeomorphic to $S^2$. 
Let $h>0$
be its degree. Suppose $z_1,\ldots,z_N$ are its singular points and let $K_1,\ldots,K_N$ be the corresponding links of singularities.
Set $K=K_1\#\ldots\# K_N$. We have the following classical result (see \cite{Moe08,Wa}), known as the \emph{genus formula}.
\begin{proposition}\label{prop:genusformula}
The three--genus of $K$ is equal to $\frac12(h-1)(h-2)$.
\end{proposition}

Now let $N$ be a tubular neighborhood of $C$ in $\CP^2$, let $W=\CP^2\setminus N$ and let $M=\partial N=-\partial W$. The following
results are proved in \cite[Section 3]{BL}:
\begin{proposition}\label{prop:onrational}
\item[(a)] $M$ is the result of $h^2$ surgery along $K$;
\item[(b)] $W$ is a rational homology sphere;
\item[(c)] $H_2(W;\Z)=0$, $H_1(W;\Z)=\Z_h$, $H_2(W,M;\Z)=\Z_h$ and $H_2(W,M;\Z)=0$.
\end{proposition}

The manifold $M$, as surgery on a knot in $S^3$, has the following enumeration of \spinc{} structures. For any integer $m\in[-h^2/2,h^2/2)$
there is a unique \spinc{} structure on $M$, denoted $\sss_m$ that extends to a \spinc{} structure $\sst_m$ on $N$, where $\sst_m$
is characterized by the fact that $\langle c_1(\sst_m),C\rangle+h^2=2m$. The enumeration we discuss here agrees with the one we
mentioned in Section~\ref{sec:undandovd}.

The \spinc{} structure $\sss_0$ is actually a Spin structure. Likewise, if $m$ is even, the \spinc{} structure corresponding to $m=-h^2/2$
is also a Spin structure. We will mostly be interested in the Spin structure $\sss_0$. Our main technical result is the following. 
\begin{proposition}\label{prop:stiefel}
Suppose that $h$ is odd. Then $H_2(W;\Z_2)=0$. In particular, $d(M,\sss_0)=\ovd(M,\sss_0)=\und(M,\sss_0)=0$.
\end{proposition}
\begin{proof}
The fact that $H_2(W;\Z_2)=0$ can be calculated 
from Proposition~\ref{prop:onrational} using the Bockstein exact sequence related to 
the short exact sequence $0\to\Z\to\Z\to\Z_2\to 0$. As $H_2(W;\Z)=0$, we have
\begin{equation}\label{eq:bockstein}
0\to H_2(W;\Z_2)\to H_1(W;\Z)\stackrel{\cdot 2}{\to} H_1(W;\Z)\to\ldots
\end{equation}
Now $H_1(W;\Z)\cong\Z_h$ and as $h$ is odd, multiplication by $2$ is an isomorphism. Hence $H_2(W;\Z_2)=0$.

By \cite[Exercise 5.6.2]{GS} the obstruction to extending a Spin structure from $M$ to $W$ is the second relative Stiefel--Whitney class
$w_2\in H^2(W,M;\Z_2)\cong H_2(W;\Z_2)$. As $H_2(W;\Z_2)=0$, any Spin structure on $M$ extends over $W$.

The result now follows by Theorem~\ref{thm:hm-main}.
\end{proof}
\begin{remark}\label{rem:evenbockstein}
If $h$ is even, by the Bockstein exact sequence \eqref{eq:bockstein} we have that $H_2(W;\Z_2)\cong \Z_2$, so the Spin structure $\sss_0$
% adding sss_0 is needed, because there might be another Spinc structure and this might actually extend.
on $M$ does not necessarily have to extend over $W$. In fact it does not extend. We discuss the case $h$ even
in detail in Section~\ref{sec:evendegree}.
\end{remark}

%
% ---------------------------------------------
%---------
\subsection{Involutive Floer homology and rational cuspidal curves}
Propositions~\ref{prop:onrational} and \ref{prop:stiefel} combined with the surgery formula \eqref{eq:ovd} (note
that the surgery coefficient $h^2$ is greater than the genus $\frac12(h-1)(h-2)$) imply the following result.
\begin{theorem}\label{thm:odddegree}
Suppose $C$ is a rational cuspidal curve of odd degree.
Let $K$ be the connected sum of the links of its singular points.
Then $K$ is involutively simple.
\end{theorem}

Together with Theorem~\ref{thm:oddstretch} we obtain the following result. 
\begin{theorem}\label{thm:main1}
Suppose $C$ is a rational cuspidal curve with singular points $z_1,\ldots,z_n$ and $n>1$. Let $K_1,\ldots,K_n$ be the corresponding links of singularities.
Assume $K_1$ is odd. Then $g(K_2\#\ldots\# K_n)\ge stretch(K_1)$.
\end{theorem}
If $K_1$ is a torus knot, Theorem~\ref{thm:main1} can be reformulated in the following way, using Theorem~\ref{thm:sch}.

\begin{theorem}\label{thm:main2}
Suppose $C$ is a rational cuspidal curve with singular points $z_1,\ldots,z_n$, $n>1$. Let $\delta_1,\ldots,\delta_n$ be the $\delta$-invariants. Assume that $z_1$ has Puiseux sequence $(p;q)$ and that $\delta_1\in S_1$, where $S_1$ is the semigroup of $z_1$. Write the regular continued fraction $\frac{q}{p}=[a_0,a_1,\ldots,a_k]$ with $a_k>1$. Then $\delta_2+\ldots+\delta_n>\intfrac{a_k-1}{2}$.
\end{theorem}

Using Theorem~\ref{thm:sumoftwo} in conjunction with Theorem~\ref{thm:odddegree} we obtain the following result.
\begin{theorem}\label{thm:withconjecture}
Suppose $C$ is a rational cuspidal curve of odd degree with precisely two singular points $z_1$ and $z_2$. Let $K_1$ and $K_2$
be links of singularities. Then $V_0(K_1\# K_2)=V_0(K_1)+V_0(K_2)$. 
In particular
\[\min_{i+j=g_1+g_2}(\#S_1\cap[0,i)+\#S_2\cap[0,j))=\#S_1\cap[0,g_1)+\#S_2\cap[0,g_2),\]
where $S_1$ and $S_2$ are the semigroups of the singular points $z_1$ and $z_2$.
\end{theorem}
\noindent Combining Theorem \ref{thm:withconjecture} with Theorem~\ref{thm:twoodds} then gives the following result.
\begin{theorem}\label{thm:sumoftwoodds}
Let $C$ be a rational cuspidal curve of odd degree with precisely two singular points $z_1$ and $z_2$. Let $K_1$ and $K_2$ be links
of singularities of $z_1$ and $z_2$. Then at least one of $K_1$ and $K_2$ is an even L--space knot.
\end{theorem}

\section{Examples and discussions}\label{sec:moreexamples}
In Sections~\ref{sec:degree5}, \ref{sec:degree7} and \ref{sec:degree9+} we
compare the new criterion (Theorem~\ref{thm:withconjecture}) with two topological results. The semigroup distribution property
conjectured in \cite{FLMN06} was established in \cite{BL} via Heegaard Floer theory, 
so it is natural to ask to what extent Theorem~\ref{thm:withconjecture} is stronger than the results of \cite{BL}.
Our results can also be compared with the spectrum semicontinuity, which is a more classical tool. 
We apply the spectrum semicontinuity via the $SS_l$ property of \cite{FLMN04}. We do not state either of the
two obstructions explicitly, referring to \cite{BL,FLMN04} instead.

We refer also to \cite{FLMN04, Moe08} for a survey of techniques for obstructing rational cuspidal curves before the Heegaard Floer obstruction.

\subsection{Degree $5$ curves with two singular points}\label{sec:degree5}
Rational cuspidal curves in $\CP^2$ with two singular points and degree $5$ have already been classified; see \cite[Section 6.1.3]{Moe08}.
We will show that topological methods are enough achieve the `geographical' part of the classification. That is, we can
show which configurations of singular points cannot be realized as singular points on a degree 5 rational cuspidal curve in $\CP^2$. All the
remaining cases can be constructed. The current Heegaard Floer methods are unable to distinguish different rational cuspidal curves with
the same configurations of singular points; that is, we cannot say anything about the `botany' problem.

For degree $5$ there are six potential configurations of pairs of singular points such that the sum of the genera of the links is $\frac12(5-1)(5-2)=6$,
that is, the genus formula (Proposition~\ref{prop:genusformula}) is satisfied. These pairs
are   
$((3;4),(3;4))$, $((3;4),(2;7))$, $((2;7),(2;7))$, $((3;5),(2;5))$, $((2;9),(2;5))$ and $((2;11),(2;3))$. Out of these, only the first case
fails the semigroup distribution property of the first author and Livingston \cite{BL}. Theorem~\ref{thm:withconjecture} obstructs
the cases $((2;7),(2;7))$ and $((2;11),(2;3))$. The remaining four cases can be realized by an explicit construction.

\begin{remark}
The case $((2;7),(2;7))$ can also be obstructed by the spectrum semicontinuity property $SS_l$; see \cite{FLMN04} for more details. The case
$((2;11),(2;3))$ cannot.
\end{remark}

\subsection{Degree $7$ curves}\label{sec:degree7}
 
For degree $7$ curves with two singular points there are altogether 41 potential configurations of singular points whose sum of genera
is $\frac12(7-1)(7-2)=15$. Out of them, 13 satisfy the semigroup distribution property. They are presented in the following
tabularized form.

\smallskip
\begin{tabular}{|c|c|l|}\hline
\multicolumn{2}{|c|}{Singular points} & Comments \\\hline
$ (3; 11)  $&$(2; 11)$ & Obstructed by Theorem~\ref{thm:withconjecture} \\\hline
$ (4; 7)   $&$(2; 13)$ & Obstructed by Theorem~\ref{thm:withconjecture} \\\hline
$ (4; 6, 13)$&$(3; 5) $ & Obstructed by Theorem~\ref{thm:withconjecture} \\\hline
$ (5; 7)   $&$(2; 7) $ & Case 3 of Fenske's list with $a=d=2$\\\hline 
$ (4; 6, 9) $&$(3; 7) $ & Case 4 of Fenske's list with $a=d=2$\\\hline
$ (4; 6, 7) $&$(3; 8) $ & Case 5 of Fenske's list with $a=d=2$\\\hline
$ (5; 6)   $&$(2; 11)$  & Case 8 of Fenske's list with $a=2$\\\hline
$ (4; 7)    $&$(3; 7) $ & Case 3 of Fenske's list with $a=1$, $d=3$\\\hline
$ (4; 5)    $&$(3; 10)$ & Case 4 of Fenske's list with $a=1$, $d=3$\\\hline
$ (4; 6, 15)$&$(3; 4) $ & Can be obstructed using \cite[Theorem 1.2]{Fen}\\\hline 
$(3; 14)$ &$(2; 5)$    &  Obstructed. See Remark~\ref{rem:moe}\\\hline
$ (3; 13)  $&$(2; 7) $ &  Obstructed. See Remark~\ref{rem:moe}\\\hline
$ (3; 10)  $&$(2; 13)$ &  Exists. See Remark~\ref{rem:moe}\\\hline
\end{tabular}

\smallskip
Here `Fenske's list' refers to the construction of Fenske \cite[Theorem 1.1]{Fen}; see also \cite[Section 7.3]{Moe08}.
The last four cases cannot be obstructed using known topological methods; however the case $((4;6,15),(3;4))$ can be obstructed
using methods from algebraic geometry.

\begin{remark}\label{rem:moe}
We were informed by Karoline Moe that a rational cuspidal curve with singular points  $(3;10)$ and $(2;13)$ can be explictly constructed and the 
two remaining cases can also be obstructed using methods of algebraic geometry. 
\end{remark}

\subsection{Some statistics on higher degree curves}\label{sec:degree9+}
It is possible to implement the semigroup distribution property of \cite{BL}, the spectrum semicontinuity and Theorem~\ref{thm:withconjecture}
and see in how many cases Theorem~\ref{thm:withconjecture} provides an obstruction. We gather calculations for low degree in the following table.

\begin{tabular}{|c|c|c|c|}\hline
Degree & Total & Pass semigroup and spectrum & Pass Theorem~\ref{thm:withconjecture}\\\hline
5& 6 & 4 & 3\\\hline
7& 41 & 13 & 10\\\hline
9& 222 & 37 & 25 \\\hline
11& 937 & 43 & 26\\\hline
13& 3539 & 90 & 66 \\\hline
15& 11925  &126& 75\\\hline
17& 35986 & 149 & 86 \\\hline
\end{tabular}
\ 

\medskip
\noindent Here `total' means the total number of pairs of singular points that pass the genus formula (Proposition~\ref{prop:genusformula}).
The third column tells how many of these pairs pass both the $SS_l$ spectrum semicontinuity property of \cite{FLMN04} and
the semigroup distribution property of \cite{BL}. 
The last column describes the number of pairs passing both the semigroup distribution and the spectrum semicontinuity obstructions
and satisfying the criterion of Theorem~\ref{thm:withconjecture}. 

\subsection{Sums of more L--space knots}

We will now give a few examples showing that the direct 
analogs of Theorems~\ref{thm:sumoftwo} and \ref{thm:sumoftwoov}  for sums of more than two L--space knots do not hold.

The first result shows the failure of Theorem~\ref{thm:sumoftwo} (and also the failure of Theorem \ref{thm:sumoftwoodds})
for sums of more than two L--space knots.
\begin{example}
It is not true in general that a sum of more than two odd L--space knots is not involutively simple. A remarkable example
is the sum $K=T(2,7)\# T(2,3)\# T(2,3)\# T(2,3)$. All the summands are odd L--space knots. We have that $V_0(T(2,3))=1$ and $V_0(T(2,7))=2$ while $V_0(K)=3$.  The knot $K$ is alternating, so the fact that
$V_0(K)=\ovV_0(K)=\unV_0(K)$ can be checked using \cite[Proposition 8.2]{HM}. However another argument can be given. There
exists a rational cuspidal curve of degree 5, with four singular points, such that the link of the first singular points is $T(2,7)$ and the links
of the remaining three are $T(2,3)$. By Theorem~\ref{thm:odddegree} we conclude that $K$ is involutively simple.
\end{example}
Another aspect of the above example is that the sum of two non-involutively simple knots (in this case, $T(2,7)\# T(2,3)$ and $T(2,3)\# T(2,3)$)
can be involutively simple.

\smallskip
Theorem~\ref{thm:sumoftwoov} also fails for sums of more than two knots. Consider the sum of three trefoils. It is an alternating
knot with Ozsv\'ath--Szab\'o $\tau$ invariant equal to $3$. The Alexander polynomial is $\Delta=(t-1+t^{-1})^3$. Write $\Delta$ as
\[\Delta=t^3-t^2+t-1+t^{-1}-t^{-2}+t^{-3}-r(t)(t-2+t^{-1}),\]
where in this case we have
\[r(t)=-1+2(t+t^{-1}).\]
By \cite[Proposition 8.2(2b)]{HM} it follows that $\ovV_0(K)<V_0(K)$.

\subsection{Relation to new conjectures of \cite{BodN}}\label{sec:differences}

In \cite[Conjecture 2.1.4]{BodN} Bodn\'ar and N\'emethi stated a conjecture about the semigroups of singular points occurring on
a rational cuspidal curve. The most natural formulation is in terms of lattice homologies of $S^3_{-h}(K)$, where $K$ is the sum
of the links of singular points and $h$ is the degree of the curve;  
see \cite[Conjecture 3.2.2]{BodN}. We compare this conjecture
with results of the present article.

There are two differences that can be seen immediately. 
First of all, all theorems in the present paper work only for rational cuspidal curve of odd degree and we
have counterexamples in even degree; see Section~\ref{sec:evendegree} below. 
The conjecture of Bodn\'ar and N\'emethi does not have this restriction.

On the other hand, the conjecture of Bodn\'ar and N\'emethi does not give any more information than the semigroup distribution property of \cite{BL}
if the number of singular points is $2$. On the contrary, results of the present article are most transparent if the number of singular points
is actually $2$. 

To conclude, the results of the present article are different than the conjecture of \cite{BodN}. 

%\begin{remark}
%Notice that the original idea of using Seiberg--Witten/Heegaard Floer
%theory to rational cuspidal curves was to study the Seiberg--Witten invariants of $S^3_{-h}(K)$, where $K$ is
%the connected sum of links of singularities of a rational cuspidal curve and $h$ is its degree.
%This approach is motivated by the theory of  superisolated surface singularities; see \cite{FLMN06,Luengo}.
%
%The approach of \cite{BL} and of
%the present article is to study $S^3_{h^2}(K)$, which is a surgery with positive coefficients. The conceptual difference between
%the two approaches gives results surprisingly similar in nature. 
%\end{remark}

\subsection{The case of even degree}\label{sec:evendegree}
As we saw already in Remark~\ref{rem:evenbockstein}, if $h=\deg C$ is even, then $H_2(W;\Z_2)\cong \Z_2$ and the Spin structure
$\sss_0$ on $M$ does not necessarily extend over $W$. In fact things are as bad as one can imagine.
\begin{proposition}\ \label{prop:evenisbad} % this label is definitely not politically correct.
Let $C$ be a rational cuspidal curve of even degree $h$, with $M$ and $W$ as in Section \ref{sec:rationalcompl}. 
\begin{itemize}
\item[(a)] The Spin structure $\sss_0$ over $M$ does not extend over $W$; moreover, the value $d(M,\sss_0)$ does not depend on the degree $h$
only.
\item[(b)] If $K$ is the connected sum of links of singular points of a rational cuspidal curve of even degree, then $K$ does not have
to be involutively simple.
\end{itemize}
\end{proposition}
The remainder of Section~\ref{sec:evendegree} is devoted to the proof of Proposition~\ref{prop:evenisbad}.

Part (a) can easily be proved using obstruction theory; here we give a quick Heegaard Floer argument. 
Notice that $d(M,\sss_0)=\frac{h^2-1}{4}-2V_0(K)$. If $\deg C$ is even, then $d(M,\sss_0)$ cannot be integral, as $V_0(K)\in\Z$.
If the Spin structure $\sss_0$ on $M$ extended to a \spinc{} structure on $W$, then by the results of Ozsv\'ath and Szab\'o \cite{OzSz03}, 
we would have $d(M,\sss_0)=0$, as $W$ is a rational homology ball. This shows that $\sss_0$ does not extend over $W$.

To show that the value of $d(M,\sss_0)$ does not depend on the degree of $C$ only, consider 
two examples depending on a parameter $h$, which we assume to be an even integer greater than $2$.
\begin{itemize}
\item[(a1)] The rational cuspidal curve of degree $h$ with singularity $(h-1;h)$.
\item[(a2)] The rational cuspidal curve of degree $h$ with singularity $(h/2;2h-1)$.
\end{itemize}
Both families exist (see \cite{FLMN04,Moe08}) and the curves can be given by an explicit equation.
The curves each have a single singular point: in case (a1) this is the torus knot $T(h-1,h)$ and in case (a2)
this is the torus knot $T(\frac{h}{2},2h-1)$. 
The genus in both cases is equal to $\frac12(h-1)(h-2)$. The $V_0$ invariant of the torus knot $T(h-1,h)$ is equal to $\frac18 h(h+2)$.
The $V_0$ invariant of the torus knot $T(\frac{h}{2},2h-1)$ is equal to $h^2/8$ if $4 \mid h$ and $\frac{h^2+1}{8}$ if $4\nmid h$. We see that the
difference between the $d$-invariants of $M$ in the two cases grows like $\frac{h}{2}$ as $h\to\infty$.

As for (b) let $C$ be the degree $6$ curve having two singular points, one with Puiseux sequence $(4;6,9)$, the other $(2;3)$. This
is Case~1 of Fenske's list with $d=a=2$ and $b=1$; see \cite{Fen} or \cite[Section 7.3]{Moe08}. The link of the $(2;3)$ singularity
is the trefoil knot and it is an odd L--space knot. The link of the $(4;6,9)$ singularity is the $(2;15)$ cable on the trefoil;
see \cite{EN} for an algorithm for determining the link of a singular point from its Puiseux sequence. As described in \cite{Wa}, the semigroup is generated by $(4,6,15)$ (the number $15$ appearing
here arises as $4\cdot 6/\gcd(4,6)+(9-6)$). The genus of the link is $9$, which does not belong to the semigroup generated
by $4,6$ and $15$. It follows that the $(2;15)$ cable on the trefoil is also an odd L--space knot. If $K$ is a connected sum
of the links of singularities of $C$, then by Theorem~\ref{thm:twoodds} $K$ is not involutively simple, so 
the analog of Theorem~\ref{thm:odddegree} for rational cuspidal curves of even degree
does not hold. 
%
%
%
% ---------------------------------------------
\section{Proof of Theorem~\ref{thm:oddstretch}}\label{sec:proofoddstretch}
For the reader's convenience we recall the statement of Theorem~\ref{thm:oddstretch}.

\newtheorem*{thm:oddstretch}{Theorem \ref{thm:oddstretch}}
\begin{thm:oddstretch}
Suppose $K_1$ is an odd L--space knot with stretch $s$. Suppose $K_2$ is an L--space knot, or a connected sum of L--space knots, with genus~$g$.
If $g<s$, then $K=K_1\#K_2$ is not involutively simple.
\end{thm:oddstretch}

The remainder of the section is devoted to the proof of this theorem.

Let $v=V_0(K_1)$. Denote by $x_0,\ldots,x_n$ the $x$--type generators of the staircase of $K_1$. Set $l=(n-1)/2$.
By Corollary~\ref{cor:easy} $K_1$ has a generator $x_l$ at bifiltration level $(v-s,v)$ and a generator $x_{l+1}$ at 
bifiltration level $(v,v-s)$. 

Let $K_2$ be a connected sum of L--space knots $J_1, \dots, J_m$ where $\cfk^\infty(J_i)$ is a staircase complex with $x$--type generators 
$x^i_0, \dots, x^i_{n_i}$.  Let $\{x'_k\}$ denote the $\prod_{i=1}^{m} (n_i+1)$--element set consisting of the products $x^1_{j_1} \otimes \dots \otimes x^m_{j_m}$, ordered in such a way that $x'_0=x^1_0 \otimes \dots \otimes x^m_0$ and $x'_1=x^1_{n_1} \otimes \dots \otimes x^m_{n_m}$. Note that $(\alpha(x'_0), \beta(x'_0)) = (0, g)$ and $(\alpha(x'_1), \beta(x'_1)) = (g, 0)$. We also have that $0 \le \alpha(x'_k), \beta(x'_k) \le g$, and that $x'_0$ (respectively $x'_1$) is the unique $x'_k$ with $\alpha(x'_k)=0 $ (respectively $\beta(x'_k)=0$).

There is a unique non-zero homology class of $\hfk^\infty(K_1\# K_2)$ in grading zero. Representatives of this homology class are of the form  $\sum_{k\in I, |I| \textup{ odd}} x_{i_k}\otimes x'_{j_k}$ plus a boundary, and we may assume that the boundary does not contain any terms of the form $x_i \otimes x'_j$. Note that adding such a boundary to $\sum_{k\in I} x_{i_k}\otimes x'_{j_k}$ cannot decrease the $\alpha$-- or $\beta$--filtration level of the resulting generator; it can only fix or increase the filtration level.

We will use Proposition~\ref{prop:C1C2}. The rest of the proof is done in 3 steps.

\subsection*{Step 1} We will show that
\[V_0(K)=v.\]
Indeed, the generator $x_l\otimes x'_1$ is at bifiltration level $(v-s+g,v)$, hence $V_0(K)\le v$. We will show that for all other elements of type $x_i\otimes x'_j$ we have 
$\max(\alpha(x_i\otimes x'_j),\beta(x_i\otimes x'_j))\ge v$. Suppose $i\le l$. Then $\alpha(x_i)\le v-s$, but $\beta(x_i)\ge v$.
Then $\beta(x_i\otimes x'_j)\ge v$ for all $j$. If $i>l$, an analogous argument shows that $\alpha(x_i\otimes x'_j)\ge v$.

\subsection*{Step 2.} We want to show that (C1) of Proposition \ref{prop:C1C2} is impossible. We aim 
to show that no element of type $\sum_{k\in I, |I| \textup{ odd}}x_{i_k}\otimes x'_{j_k}$ plus a boundary has bifiltration level $(v,v)$.

We first show that if an element of the form $x_i \otimes x'_j$ has $\max(\alpha(x_i \otimes x'_j), \beta(x_i \otimes x'_j)) \le v$, then $i=l $ or $l+1$. Consider $x_i \otimes x'_j$, and suppose that $i <l $. Then $\beta(x_i)\ge v+1$ and so $\beta( x_{i}\otimes x'_{j}) \ge v+1$. An analogous argument works if $i>l+1$. Hence $i=l$ or $l+1$. Note that if $i=l$, then $\beta(x_i \otimes x'_j) \le v$ if and only if $\beta(x'_j)=0$. Analogously, if $i=l+1$,  then $\alpha(x_i \otimes x'_j) \le v$ if and only if $\alpha(x'_j)=0$. In particular, if $x_i \otimes x'_j$ has $\max(\alpha(x_i \otimes x'_j), \beta(x_i \otimes x'_j)) \le v$, then $x_i \otimes x'_j$ is either $x_l \otimes x'_1$ or $x_{l+1} \otimes x'_0$. 

We have the following lemma, whose proof is given below. 

\begin{lemma}\label{lem:zerogr}
If $z \in \cfk^\infty_0(K_1 \# K_2)$ is a boundary, then either $z = x_l \otimes x'_1 + x_{l+1} \otimes x'_0$ or $\max(\alpha(z), \beta(z)) > v$.
\end{lemma}

\noindent We have that $(\alpha(x_l \otimes x'_1), \beta(x_l \otimes x'_1)) = (v-s+g, v)$. By Lemma \ref{lem:zerogr}, if we add a boundary (other than $x_l \otimes x'_1+x_{l+1} \otimes x'_0$) of grading zero to $x_l \otimes x'_1$, then either the $\alpha$-- or $\beta$--filtration level of the result will be strictly larger than $v$. A similar argument applies to $x_{l+1} \otimes x'_0$. Hence no generator for $HKF^\infty_0(K_1 \# K_2)$  has bifiltration level $(v,v)$. On the other hand (C1) implies existence of such a generator. Therefore (C1) cannot hold.

\subsection*{Step 3.} 
We will now show that (C2) does not happen. Indeed, the result follows from the following lemma:

\begin{lemma}\label{lem:onegr}
If $z \in \cfk^\infty_1(K_1 \# K_2)$, then $\max(\alpha(z), \beta(z)) > v$.
\end{lemma}

\noindent The proof of Lemma~\ref{lem:onegr} is given at the end of this section. The lemma implies that 
there is no $z$ with $\max(\alpha(z), \beta(z)) \le v$ such that $\partial z = (1+\iota)x$ for $x$ a generator of homology with 
$\max(\alpha(x), \beta(x))=v$, and thus (C2) is impossible. This concludes the proof that the knot $K$ is not involutively simple.

\begin{proof}[Proof of Lemma~\ref{lem:zerogr}]%\label{sec:proofofzerogr}
As above, let $K_2 = J_1 \# \dots \# J_m$ and for notational convenience, we rewrite $K_1$ as $J_0$. Then $K_1 \# K_2 = J_0 \# J_1 \# \dots \# J_m$. An element of $\cfk^\infty_0(J_0 \# J_1 \# \dots \# J_m)$ is a sum of elements of the form  $U^\frac{N}{2} z^0 \otimes z^1 \otimes \dots \otimes z^m$, where each $z^i$ is either a $x$--type or $y$--type generator of $\cfk^\infty(J_i)$ and $N$ even is the number of $y$--type generators in the product. Let $\boldsymbol{z}$ denote $z^0 \otimes z^1 \otimes \dots \otimes z^m$.

We first consider the case when all of the $z^i$ are $x$--type generators, i.e., $N=0$. A sum of elements of the form $x^0_{i_0} \otimes \dots \otimes x^m_{i_m}$ is a boundary if and only if the number of terms is even. Recall from above that $\{x'_k\}$ is the set consisting of the products $x^1_{j_1} \otimes \dots \otimes x^m_{j_m}$, ordered such that $x'_0=x^1_0 \otimes \dots \otimes x^m_0$ and $x'_1=x^1_{n_1} \otimes \dots \otimes x^m_{n_m}$. Recall that $\alpha(x'_k)$ and $\beta(x'_k)$ are both non-negative, and $x'_0$ (respectively $x'_1$) is the unique $x'_k$ with $\alpha(x'_k)=0 $ (respectively $\beta(x'_k)=0$). Also, recall that $(\alpha(x^0_l), \beta(x^0_l))=(v-s,v)$ and $(\alpha(x^0_{l+1}), \beta(x^0_{l+1}))=(v,v-s)$ and that $\max(\alpha(x^0_i),\beta(x^0_i))>v$ for $i \neq l, l+1$. It follows that if $N=0$, the only boundary with $\max(\alpha,\beta)\le v$ is $x_l \otimes x'_1 + x_{l+1} \otimes x'_0$.

Next, we consider the general case where a positive even number $N$ of the $z^i$ are of $y$--type. If $z^0$ is of $y$--type, then $\max(\alpha(z^0), \beta(z^0)) \ge v$, with equality if and only if $z^0=y^0_l$, and $N-1$ of $z^1, \dots, z^m$ are of $y$--type. Since $y$--type generators satisfy $\min(\alpha(y), \beta(y)) \ge 1$, and $x$--type generators satisfy $\min(\alpha(x), \beta(x)) \ge 0$, it follows that a product $\boldsymbol{z}$ of $y^0_i$ with $N-1$ $y$--type generators (and the remaining $m-N+1$ factors of $x$--type) will have $\max(\alpha(\boldsymbol{z}), \beta(\boldsymbol{z})) \ge v+N-1$. Multiplication by $U$ lowers both $\alpha$ and $\beta$ by $1$ and so $\max(\alpha(U^\frac{N}{2}\boldsymbol{z}), \beta(U^\frac{N}{2}\boldsymbol{z}))\ge v-1+\frac{N}{2} \geq v$. If the second inequality is strict, we are done. If the second inequality is actually an equality, then we must have that $z^0=y^0_l$ and $N=2$. Now, if $z^0=y^0_l$, then $(\alpha(z^0), \beta(z^0))=(v,v)$. If $N=2$, then $m \ge 2$, since there must be at least one factor of $x$--type in order for $z^0 \otimes z^1 \otimes \dots \otimes z^m$ to be a term in a boundary. Since $\min(\alpha(y), \beta(y)) \ge 1$ and $\max(\alpha(x),\beta(x))>0$, it follows that $\max(\alpha(\boldsymbol{z}), \beta(\boldsymbol{z})) \ge v+N$, equivalently $\max(\alpha(U^\frac{N}{2}\boldsymbol{z}), \beta(U^\frac{N}{2}\boldsymbol{z}))\ge v+\frac{N}{2}$, as desired.

If $z^0$ is of $x$--type, then $\max(\alpha(z^0), \beta(z^0)) \ge v$, and $N$ of $z^1, \dots, z^m$ are of $y$--type. Since $\min(\alpha(y), \beta(y)) \ge 1$, it follows that $\max(\alpha(\boldsymbol{z}), \beta(\boldsymbol{z})) \ge v+N$, hence $\max(\alpha(U^\frac{N}{2}\boldsymbol{z}), \beta(U^\frac{N}{2}\boldsymbol{z}))\ge v+\frac{N}{2}$.
\end{proof}

\begin{proof}[Proof of Lemma \ref{lem:onegr}]%\label{sec:proofofonegr}
Again, let $K_2 = J_1 \# \dots \# J_m$ and $K_1=J_0$ so that $K_1 \# K_2 = J_0 \# J_1 \# \dots \# J_m$. An element of $\cfk^\infty_1(J_0 \# J_1 \# \dots \# J_m)$ is a sum of elements of the form  $U^\frac{N-1}{2} z^0 \otimes z^1 \otimes \dots \otimes z^m$, where each $z^i$ is either a $x$--type or $y$--type generator of $\cfk^\infty(J_i)$ and $N$ odd is the number of $y$--type generators in the product. As before, let $\boldsymbol{z}$ denote $z^0 \otimes z^1 \otimes \dots \otimes z^m$.

We first consider the case when exactly one $z^i$ is a $y$--type generator. Consider $y^0_1, \dots, y^0_{n-1}$, the $y$--type generators of $\cfk^\infty(J_0)$. Suppose that $z^0 \in \{y^0_1, \dots, y^0_{n-1} \}$. Note that for $i \neq l,l+1$, we have $\max(\alpha(y^0_i), \beta(y^0_i)) > v$ and we are done. Thus, suppose $z^0=y^0_l$. (The case $z^0=y^0_{l+1}$ is analogous.) Recall that $(\alpha(y^0_l),\beta(y^0_l))=(v, v)$. Since $x$--type generators satisfy $\max(\alpha(x), \beta(x)) > 0$ and $\min(\alpha(x), \beta(x)) \ge 0$, we have that the product $y^0_l \otimes x^1_{i_1} \otimes \dots \otimes x^m_{i_m}$ satisfies $\max(\alpha, \beta)>v$, as desired.
We now consider the case where $z^0 \in \{ x^0_0, \dots, x^0_{n}\}$. For $i \neq l, l+1$, we have that $\max(\alpha(x^0_i), \beta(x^0_i)) > v$. Hence $z^0$ must be either $x^0_l$ or $x^0_{l+1}$. We consider the case when $z^0=x^0_l$; the other case is similar. Recall that $(\alpha(x^0_l), \beta(x^0_l))=(v-s, v)$. Recall that for a $y$--type generator, we have that $\min(\alpha(y), \beta(y)) \ge 1$, and for a $x$--type generator, $\min(\alpha(x), \beta(x)) \ge 0$. Exactly one of $z^1, \dots, z^m$ is of $y$--type, and so
\[ \max(\alpha(x^0_l \otimes z^1 \otimes \dots \otimes z^m), \beta(x^0_l \otimes z^1 \otimes \dots \otimes z^m)) > v. \]

We now consider the general case where an odd number $N$ of the $z^i$ are of $y$--type and $N\ge3$. If $z^0$ is of $y$--type, then we have $\max(\alpha(z^0), \beta(z^0)) \ge v$ and $N-1$ of $z^1, \dots, z^m$ are of $y$--type. Since $y$--type generators satisfy $\min(\alpha(y), \beta(y)) \ge 1$, and $x$--type generators satisfy $\min(\alpha(x), \beta(x)) \ge 0$, it follows that a product of $y^0_i$ with $N-1$ $y$--type generators (and the remaining factors of $x$--type) will have $\max(\alpha(\boldsymbol{z}), \beta(\boldsymbol{z})) \ge v+N-1$. Multiplication by $U$ lowers both $\alpha$ and $\beta$ by $1$ and so $\max(\alpha(U^\frac{N-1}{2}\boldsymbol{z}), \beta(U^\frac{N-1}{2}\boldsymbol{z}))\ge v+\frac{N-1}{2}$.

If $z^0$ is of $x$--type, then $\max(\alpha(z^0), \beta(z^0)) \ge v$ and $N$ of $z^1, \dots, z^m$ are of $y$--type. Then the product of $x^1_i$ with $N$ generators of $y$--type will have $\max(\alpha(\boldsymbol{z}), \beta(\boldsymbol{z})) \ge v+N$. It follows that $\max(\alpha(U^\frac{N-1}{2}\boldsymbol{z}), \beta(U^\frac{N-1}{2}\boldsymbol{z}))\ge v+\frac{N+1}{2}$.
\end{proof}

%
%
%
%
%
%
%
%
% ---------------------------------------------------------------------

\section{Proof of Theorem~\ref{thm:sumoftwo}}\label{sec:proofoftwo}
For the reader's convenience we recall the statement of Theorem~\ref{thm:sumoftwo}.

\newtheorem*{thm:sumoftwo}{Theorem \ref{thm:sumoftwo}}
\begin{thm:sumoftwo}
Let $K_1$ and $K_2$ be L--space knots and set $K=K_1\# K_2$. If $V_0(K)<V_0(K_1)+V_0(K_2)$, then $V_0(K)$ is not involutively simple.
\end{thm:sumoftwo}
The proof is done in nine steps and takes the remainder of the section. Throughout the proof we will assume that 
the staircase for $K_1$ is generated by elements $x_0,\ldots,x_n$, $y_0,\ldots,y_{n-1}$
such that $\partial y_i=x_i+x_{i+1}$ and the staircase for $K_2$ is generated by elements $x'_0,\ldots,x'_m$, $y'_0,\ldots,y'_{m-1}$.

\subsection*{Step 1. Structure of generators of $\cfk^\infty_0(K)$.}
We first investigate the structure of elements in $\cfk^\infty_0(K)$ which generate 
$\hfk^\infty_0(K)$.
\begin{lemma}\label{lem:step0}
Suppose $x\in\cfk^\infty_0(K)$ is a cycle and generates the homology. Then $x$ is a sum
of an odd number of products $x_i\otimes x_j'$.
\end{lemma} 
\begin{proof}
By the definition of staircase complexes, an element in $\cfk^\infty_0$ is of the form
\[\sum_{(i,j)\in I_x} x_i\otimes x_j'+U^{-1}\sum_{(k,l)\in I_y} y_k\otimes y_l',\]
where the sums are taken over some subsets $I_x,I_y$ of indices. We know that $\partial(x_i\otimes x_j')=0$. Hence in order to check
whether an element is a cycle, we need to look at the sum of terms of type $y_k\otimes y_l'$.

Suppose $z=y_{k_1}\otimes y_{l_1}'+\ldots+y_{k_r}\otimes y_{l_r}'$ and $\partial z=0$. Write $\partial$ as the sum $\partial_1+\partial_2$,
where $\partial_1$ acts only on the first coordinate of the tensor product and $\partial_2$ acts on the second coordinate. As $\partial_1z$
is a sum of elements of type $x\otimes y'$ and $\partial_2z$ is a sum of elements of type $y\otimes x'$ it follows that $\partial z=0$
implies that $\partial_1 z=\partial_2 z=0$.

Consider the equation $\partial_1 z=0$. It implies that $z$ can be written as the sum $\wt{y}_0\otimes y_0'+\ldots+\wt{y}_m\otimes y_m'$,
where each of the $\wt{y}_j$ is a cycle in $\cfk^\infty_1(K_1)$. But $\cfk^\infty(K_1)$ is a staircase complex and 
$\partial\colon\cfk^\infty_1(K_1)\to\cfk^\infty_0(K_1)$ is injective. Therefore $\wt{y}_1=\ldots=\wt{y}_m=0$ and so $z=0$.

As a consequence, if an element $x\in\cfk^\infty_0$ is a cycle, it must be a sum of elements of type $x_i\otimes x_j'$. As each summand
is a generator of homology and a sum of generators is a generator if and only if the number of summands is odd, this concludes the proof
of the lemma.
\end{proof}

\subsection*{Step 2. Excluding the case (C1).} 
We will use Proposition~\ref{prop:C1C2}. We first exclude the possibility of case (C1). 
Suppose $x$ is a degree zero generator of the homology of 
$\cfk^\infty(K)$ such that $(1+\iota)x=0$
and $\max(\alpha(x),\beta(x))=V_0(K)$. 
By Lemma~\ref{lem:step0} $x$ is a sum of elements of type $x_i\otimes x_j'$ and the number of summands is odd.
By Theorem~\ref{thm:zemke}, $\iota x=x$ implies that there must be at least
one element of type $x_i\otimes x_j'$ which is fixed under $\iota$, that is, $\iota(x_i\otimes x_j')=x_i\otimes x_j'$.
Using Theorem~\ref{thm:zemke} again, we deduce that $2i=n$ and $2j=m$, and so both the staircases are 
even and $x_i$, $x_j'$ are middle elements of the staircases. It follows from Corollary~\ref{cor:easy}
that $x_i$ is at bifiltration level $(V_0(K_1),V_0(K_1))$ and $x_j'$ is at bifiltration level $(V_0(K_2),V_0(K_2))$. But then
$x_i\otimes x_j'$ is at bifiltration level $(V_0(K_1)+V_0(K_2),V_0(K_1)+V_0(K_2))$. However we assumed that $x$ is at bifiltration
level at most $(V_0(K),V_0(K))$. This contradicts the assumption that 
$V_0(K)<V_0(K_1)+V_0(K_2)$. So case (C1) is excluded.

\smallskip
It remains to deal with (C2). This case is more complicated and takes the remainder of the proof.

\subsection*{Step 3. Preliminaries on the case (C2).} We begin with the following lemma.
\begin{lemma}\label{lem:levelV0}
Suppose that $x$ is a generator of the homology of $\cfk^\infty_0(K)$ and $(1+\iota)x=\partial y$ for some $y=\cfk^\infty_1(K)$
such that $\max(\alpha(y),\beta(y)) \leq V_0(K)$.
Then $y$ is a sum of elements $\sum x_i\otimes y_j'+\sum y_k\otimes x_l'$ and all the summands are at bifiltration level
$(V_0(K),V_0(K))$.
\end{lemma}
\begin{proof}
First of all, the only elements in $\cfk^\infty(K)$ of odd degree are sums $\sum x_i\otimes y_j'+\sum y_k\otimes x_l'$ multiplied
by some power of $U$. So if $x$ is at the
homological grading $0$ and $(1+\iota) x=\partial y$, then $y$ must have the form as in the hypothesis of the lemma. By condition (C2) it follows
that $\alpha(y),\beta(y)\le V_0(K)$.

Suppose there is a summand in $y$ at bifiltration level $(\alpha,\beta)$ with $\alpha<V_0(K)$. Assume it is $x_i\otimes y_j'$ 
(the other case is $y_k\otimes x_l'$). By hypothesis $\beta\le V_0(K)$. But then $x_i\otimes x'_{j+1}$ has both
filtration levels strictly less than $V_0(K)$. This contradicts the definition of $V_0(K)$ given in \eqref{eq:v0}. 
\end{proof}

%\smallskip
\subsection*{Step 4. Construction of an auxiliary graph $\Gamma$.}
In light of Lemma \ref{lem:levelV0}, we need to investigate the possibility of the existence of $x,y$ such that $(1+\iota)x=\partial y$ and all of the terms in $y$ are
at bifiltration level $(V_0(K),V_0(K))$. Suppose 
\[y=x_{i_1}\otimes y_{j_1}'+\ldots+x_{i_r}\otimes y_{j_r}'+y_{k_1}\otimes x_{l_1}'+\ldots+y_{k_s}\otimes x_{l_s}'.\]
Let us define the following graph. For each summand of $y$ we take two vertices. For a summand $x_i\otimes y_j'$ the vertices
have labels $(i,j)$ and $(i,j+1)$ and they are connected by a green edge. For a summand $y_k\otimes x_l'$ the vertices have labels $(k,l)$
and $(k+1,l)$ and are connected by a blue edge. One thinks of the vertices of the graph as the elements appearing in the differential of $y$.

We will introduce two more types of edges: Red ones corresponding to cancellations of terms in the differential and orange ones that encode the action of $\iota$.

More specifically, 
pair up the vertices with the same label $(i,j)$. Connect the vertices in pairs by a red edge. If the number of vertices with the same label is odd,
one of the vertices is not connected by any red edge. We insist that each vertex be adjacent to at most one red edge.

Finally we draw orange edges. Suppose we have two distinct vertices with labels $(a,b)$ and $(n-a,m-b)$ for some $a$ and $b$ such that
none of them is adjacent to a red edge or to an orange edge. Then we connect them by an orange edge. 
We repeat the procedure until no more orange edges can be drawn. 
Let $\Gamma$ be the resulting graph.

\subsection*{Step 5. First properties of the graph $\Gamma$.} We prove the following result.
\begin{lemma}\label{lem:step4}
Let $y$ and $\Gamma$ be as above. Suppose that no vertex of $\Gamma$ is labelled $(\frac{n}{2},\frac{m}{2})$. Then
every vertex of $\Gamma$ is adjacent to precisely two edges (one green or blue, one red or orange).
\end{lemma}

\begin{remark}
If there is a vertex of $\Gamma$ with label $(\frac{n}{2},\frac{m}{2})$, then necessarily $n$ and $m$ are even and
there is a summand of $y$ such that $x_{n/2}\otimes x'_{m/2}$ is a summand of its differential. 
But $x_{n/2}\otimes x'_{m/2}$ is at bifiltration level
$(V_0(K_1)+V_0(K_2),V_0(K_1)+V_0(K_2))$. This contradicts the assumption that $V_0(K)<V_0(K_1)+V_0(K_2)$, because
$y$ is at bifiltration level $(V_0(K),V_0(K))$; compare the arguments in Step~2.
\end{remark}

\begin{proof}[Proof of Lemma \ref{lem:step4}]
It follows by construction that each vertex is adjacent either to one green edge or to one blue edge. Moreover, it is adjacent
to at most one orange or red edge. Suppose $(i,j)$ is the label of a vertex that is not adjacent to any orange or red edge.
It follows that $x_i\otimes x_j'$ appears in the differential of $y$,
but $x_{n-i}\otimes x_{m-j}'$ does not (unless $i=n-i$ and $j=m-j$). 
This contradicts the fact that $\partial y$ is of the form $(1+\iota)x$ for some $x$.
\end{proof}

From now on we will assume that each vertex of $\Gamma$ is adjacent to precisely two edges.
It follows that the graph $\Gamma$ splits as a disjoint union of cycles $\Gamma_1,\ldots,\Gamma_z$. We call a cycle
$\Gamma_t$ \emph{even} or \emph{odd} depending on whether the number of \emph{orange} edges is even or odd.

\subsection*{Step 6. Finding odd cycles in $\Gamma$.} We now exploit the assumption that $x$ is non-trivial in homology.
\begin{lemma}\label{lem:step5}
At least one of the cycles $\Gamma_1,\ldots,\Gamma_z$ is odd.
\end{lemma}
\begin{proof}
Consider $\partial y$. It is a sum of elements $x_i\otimes x'_j$. These elements correspond to vertices of $\Gamma$, and as we mentioned above,
red edges indicate elements cancelling in the differential. So the number of (non-cancelling) summands in $\partial y$ is the number of vertices not adjacent
to red edges. By Lemma~\ref{lem:step4} this is twice the number of orange edges of $\Gamma$.
Suppose the total number of the orange edges of $\Gamma$ is even. Then $\partial y$ has $4t$ summands for some $t\in\Z$. But $\partial y=(1+\iota)x$.
Now write 
\[x=x_{i_1}\otimes x_{j_1}'+\dots+x_{i_u}\otimes x_{j_u}'.\]
and assume that the summands are ordered in such a way 
that for $s=1,\ldots,w$ we have $i_{2s}=n-i_{2s-1}$ and $j_{2s}=m-j_{2s-1}$ and $w$ is maximal number
for which such ordering is possible. Note that $u$ is odd, since $x$ represents a nontrivial element in $\hfk^\infty(K)$.

If no summand of $x$ is fixed by $\iota$, then $(1+\iota)x$ will have precisely $2(u-2w)$ summands, because the first $2w$ terms of $x$
will be mutually cancelled. But since $u$ is odd, $2(u-2w)$ is not divisible by $4$, so we get a contradiction.

If there exists a summand in $x$ which is fixed by $\iota$, we repeat the argument in Step 2 to reach a contradiction.

We conclude that $\Gamma$ has an odd number of orange edges,
so at least one of the connected components must have an odd number of edges.
\end{proof}

\subsection*{Step 7. Consequences of the existence of an odd cycle.}  In the following,
if a vertex has label $(i,j)$, then we refer to $i$ as the \emph{first label} and to $j$ as the \emph{second label}.
\begin{lemma}\label{lem:step6}
If $n$ is even, then $\Gamma$ contains a vertex with label $(\frac{n}{2},j)$ for some $j$. If $n$ is odd, then $\Gamma$ contains a vertex
with label $(\frac{n-1}{2},j)$ and a vertex with label $(\frac{n+1}{2},j)$ (with the same $j$); moreover, the two vertices are connected by a blue edge.
\end{lemma}
\begin{proof}
Suppose $n$ is even and let $\Gamma_1$ be a cycle with an odd number of orange edges.  Count the edges 
connecting vertices with the first label less than $\frac{n}{2}$ on one side and greater than or equal
to $\frac{n}{2}$ on the other side.
Each orange edge enters the count and no green or red edge does, because the first labels of the vertices at both ends of a green or red edge are the same. 
Suppose no vertex of $\Gamma_1$ has
label of type $(\frac{n}{2},j)$ (for any $j$).  
This assumption implies that no blue edge counts either, because a blue edge connects vertices whose first labels differ by $1$. 
But the total number of edges that `cross
the value $\frac{n}{2}$' must be even, because $\Gamma_1$ is a cycle. This is a contradiction.

The proof for $n$ odd is essentially the same; we count edges connecting vertices having the first label less than $\frac{n}{2}$ on one
side and greater than $\frac{n}{2}$ on the other. The number of such edges must be even because $\Gamma_1$ is a cycle. As
the total number of orange edges is odd, at least one blue edge must enter the count. This blue edge must connect two
vertices with first labels $\frac{n-1}{2}$ and $\frac{n+1}{2}$, respectively. However blue edges connect vertices with the same second label,
hence $j$ is the same for both vertices.
\end{proof}

\subsection*{Step 8. Finishing the proof if $n$ is even.}

If $n$ is even, then by Lemma~\ref{lem:step6} the differential $\partial y$ contains an element $x_{n/2}\otimes x_{j}'$ for some $j$.
%The number of times $x_{n/2}\otimes x_{j}'$ appears in the differential might be even, but this does not affect the argument.
Moreover, by Corollary~\ref{cor:easy}, 
$x_{n/2}$ is at bifiltration level $(V_0(K_1),V_0(K_1))$. No matter what $j$ is, either $\alpha(x_{j}')\ge V_0(K_2)$
or $\beta(x_{j}')\ge V_0(K_2)$. This implies that $y$ has a summand with either $\alpha$-- or $\beta$-- filtration level
at least $V_0(K_1)+V_0(K_2)$. But $y$ was at bifiltration level $(V_0(K),V_0(K))$. This contradict the assumption
that $V_0(K)<V_0(K_1)+V_0(K_2)$.

\subsection*{Step 9. Finishing the proof if $n$ is odd.}

Lemma~\ref{lem:step6} implies that for some $j$ the elements $x_{(n-1)/2}\otimes x_{j}'$ and $x_{(n+1)/2}\otimes x_{j}'$ are present in $\partial y$. The
fact that the corresponding vertices in the graph were connected by a blue edge implies that $y_{(n-1)/2}\otimes x_j'$ is a summand
of $y$. But $y_{(n-1)/2}$ is at bifiltration level $(V_0(K_1),V_0(K_1))$ and $x_{j}'$ has at least one filtration level greater or
equal to $V_0(K_2)$. Then $y_{(n-1)/2}\otimes x_j'$ has at least one filtration level greater than or equal to $V_0(K_1)+V_0(K_2)$.
By Lemma~\ref{lem:levelV0} this is a contradiction with the fact that $V_0(K)<V_0(K_1)+V_0(K_2)$.

\section{Proof of Theorem~\ref{thm:sumoftwoov}}\label{sec:proofovv}
For the reader's convenience we repeat the statement of Theorem~\ref{thm:sumoftwoov}.

\newtheorem*{thm:sumoftwoov}{Theorem \ref{thm:sumoftwoov}}
\begin{thm:sumoftwoov}
Suppose $K$ is a connected sum of two L--space knots. Then $\ovV_0(K)=V_0(K)$.
\end{thm:sumoftwoov}
The proof takes the remaining part of the section.

Write $K=K_1\# K_2$. Let $x_0,\ldots,x_n$, respectively $x_0',\ldots,x_{n'}'$ be the $x$--type generators of
$St(K_1)$, respectively of $St(K_2)$.
The $y$--type generators are denoted, respectively, by $y_0,\ldots,y_{n-1}$ and $y'_0,\ldots,y'_{n'-1}$. Assume that $\ovV_0(K)<V_0(K)$.

Any element $Q(x_{i_0}\otimes x_{j_0}')\in AI^+$ generates the second tower, in the sense that $Q(x_{i_0}\otimes x_{j_0}')$ is in the image of $U^n$ for all $n \in \N$. So let us take an element $x=x_{i_0} \otimes x_{j_0}'$. Any element $z\in AI^+$ generating the second tower must be homologous to $Qx$.
By \eqref{eq:ovddef}, $\ovV_0(K)$ is the minimum of $\max(\alpha(z),\beta(z))$ taken over all the elements $z\in AI^+$ homologous to $Qx$.
Therefore, in order to see whether $\ovV_0(K)<V_0(K)$ we need to 
check whether there exists an element $z\in AI^+$ homologous to $Qx$ such that $\max(\alpha(z),\beta(z))<V_0(K)$. 
The condition that $z$ be homologous to $Qx$ means that there exists $y\in AI^+$ such that $\di y=z+Qx$. Write $y=y_{(0)}+Qy_{(1)}$
and $z=z_{(-1)}+Qz_{(0)}$. We have $y_{(1)}\in\cfk^\infty_1(K)$, $z_{(0)},y_{(0)}\in\cfk^\infty_0(K)$ and $z_{(-1)}\in\cfk^\infty_{-1}(K)$. The
condition $\di y=z+Qx$ translates into $\partial y_{(0)}=z_{(-1)}$ and $\partial y_{(1)}+(1+\iota)y_{(0)}=z_{(0)}+x$.

\begin{lemma}\label{lem:z0is0}
If $\ovV_0(K)<V_0(K)$, then  $z_{(0)}=0$.
\end{lemma}
\begin{proof}[Proof of Lemma~\ref{lem:z0is0}]
Suppose the contrary, that is,
$z_{(0)}\neq 0$. If $(1+\iota)y_{(0)}$ has only elements of type $x_i\otimes x'_j$, then $z_{(0)}$ is also a sum of such elements (because
$\partial y_{(1)}$ is and $x$ is)
and $\max(\alpha(z_{(0)}),\beta(z_{(0)}))\ge V_0(K)$, contrary to the assumption that $\max(\alpha(z), \beta(z))<V_0(K)$. Therefore $(1+\iota)y_{(0)}$ must have a summand $Uy_i\otimes y_j'$ for some $i,j$. This summand must also appear in $z_{(0)}$.
By the assumption that $\ovV_0(K)<V_0(K)$, each summand of $z_{(0)}$ must have both filtration levels less than $V_0(K)$ and so we must have that
$\max(\alpha(y_i\otimes y_j'),\beta(y_i\otimes y_j'))<V_0(K)+1$. As the horizontal differential decreases the $\alpha$ grading
at least by $1$ and the vertical differential decreases the $\beta$ grading at least by $1$, it follows that
$\max(\alpha(x_i\otimes x_{j+1}'),\beta(x_i\otimes x_{j+1}'))<V_0(K)$, which contradicts the definition of $V_0(K)$.

Thus, the only way to have $\ovV_0(K)<V_0(K)$ is for $z_{(0)}=0$. 
\end{proof}
As $z$ is homologically non-trivial, the statement that $z_{(0)}=0$ implies that
$z_{(-1)}\neq 0$. We have the following property of $z_{(-1)}$.
\begin{lemma}\label{lem:z1ispsecific}
The element 
$z_{(-1)}$ is a linear combination of elements of the form $Ux_i \otimes y_j'$ and $Uy_k\otimes x_l'$ all of which are at bifiltration level
$(V_0(K),V_0(K))$.
\end{lemma}
\begin{proof}[Proof of Lemma~\ref{lem:z1ispsecific}]
As $z_{(-1)}$ is at grading $-1$, it must be a sum
elements of the form $Ux_i\otimes y_j'$ and $Uy_k\otimes x_l'$. If $\max(\alpha(z_{(-1)}),\beta(z_{(-1)}))<V_0(K)$, then necessarily each $x_i\otimes y_j'$
and $y_k\otimes x_l'$ must be at bifiltration level at most $(V_0(K),V_0(K))$. Suppose a summand of $U^{-1}z_{(-1)}$ is \emph{not}
at bifiltration level $(V_0(K),V_0(K))$. Without loss of generality we assume that this is $x_i\otimes y_j'$ and $\alpha(x_i\otimes y_j')<V_0(K)$.
Then the vertical differential of this element, $x_i\otimes x_{j+1}'$, has both filtration levels strictly less than $V_0(K)$. This
contradicts the definition of $V_0(K)$ and the contradiction concludes the proof.
\end{proof}

\smallskip
We resume the proof of Theorem~\ref{thm:sumoftwoov}.
Write $U^{-1}y_{(0)}$ as a sum $w_y+w_x$, where $w_y$ is the sum of elements of type $y_i\otimes y_j'$ and $w_x$ is the sum of
elements of type $U^{-1}x_k\otimes x_l'$. As $\partial w_x=0$, it follows that $\partial w_y=U^{-1}z_{(-1)}$. In particular, $w_y\neq 0$.

We claim that $w_y$ must have a summand which is at bifiltration level $(V_0(K),V_0(K)+a)$ or $(V_0(K)+a,V_0(K))$ for some $a>0$. If not,
none of the summands of $\partial w_y$ can possibly be at bifiltration level $(V_0(K),V_0(K))$, contradicting Lemma~\ref{lem:z1ispsecific}. 
So assume, without loss of generality,
that $y_i\otimes y_j'$ is a summand of $w_y$ which is at bifiltration level $(V_0(K),V_0(K)+a)$. Now 
\[\partial (y_i\otimes y_j')=x_i\otimes y_j'+x_{i+1}\otimes y_j'+y_i\otimes x_j'+y_i\otimes x_{j+1}'.\]
We have $\alpha(x_i\otimes y_j')<\alpha(y_i\otimes y_j')=V_0(K)$. Hence 
the element $x_i\otimes y_j'$ cannot be at bifiltration level $(V_0(K),V_0(K))$.
Therefore $y_{i-1}\otimes y_j'$ must be a summand of $w_y$, for otherwise $x_i\otimes y_j'$ survives in $\partial w_y=U^{-1}z_{(-1)}$,
contrary to Lemma~\ref{lem:z1ispsecific}.

Clearly $\alpha(y_{i-1}\otimes y_j')<V_0(K)$. We look now at the differential $\partial (y_{i-1}\otimes y_j')$. Again by Lemma~\ref{lem:z1ispsecific}
the element 
$x_{i-1}\otimes y_j'$ must also get cancelled in $\partial w_y$.
Repeating this argument we show that $y_{0}\otimes y_j'$ must be a summand of $w_y$ and then $x_0\otimes y_j'$ must appear in $\partial w_y$.
But $x_0\otimes y_j'$ cannot be cancelled anymore. So it must appear in $\partial w_y$, yet $\alpha(x_0\otimes y_j')\le\alpha(x_i\otimes y_j')$
(equality can occur if and only if $i=0$) and $\alpha(x_i\otimes y_j')<V_0(K)$. 

The contradiction shows that $\max(\alpha(z_{(-1)}),\beta(z_{(-1)}))\ge V_0(K)$, that is, $\ovV_0(K)=V_0(K)$.

%
%
%
%
%
%
% --------------------------------------------------------------------
\appendix

%
%
%
%
%
%
% ---------------------------------------------
\section{The stretch for the torus knot $T(p,q)$\\ (by A. Schinzel)}\label{sec:proof}
Throughout the section we denote by $K(p,q)$ the maximal integer $s$ such that either
$\delta,\delta+1,\ldots,\delta+s$ all belong to the semigroup generated by $p$ and $q$, or none of these numbers belongs, where $\delta=\frac12(p-1)(q-1)$.
It follows from Sections \ref{sec:Lspaceknots} and \ref{sec:stretchLspace} that the stretch of $T(p,q)$ is equal to $1+K(p,q)$.

\begin{theorem}\label{thm:sch1} Let $q,p>1$, $\gcd(p,q)=1$ and $\frac{q}{p}=[a_0,a_1,\ldots,a_n]$ be a regular  continued fraction as
in \eqref{eq:contfracexp} with $a_n>1$.
Then $K(p,q)=\intfrac{a_n-1}{2}$.
\end{theorem}
We set up some notation. Write $q$ as $kp+r$ for $k\in\Z$ and $0<r<p$. Set $\ol{r},l$ to be integers such that
$0<\ol{r}<p$ and $r\ol{r}=1+lp$. For $x\in\R$,  let $\lfloor x\rfloor$ and $\{x\}$ denote the integral and fractional part of $x$, respectively; 
that is, $\lfloor x\rfloor$ is the greatest integer less than or equal to $x$, and $\{x\}=x-\lfloor x\rfloor$. The ceiling of $x$ is $\lceil x\rceil=x+\{-x\}$.

For a non-negative integer $\rho$, define
\begin{align*}
r_{\rho}&=p\left\{\frac{\frac{\ol{r}-1}{2}+\rho\ol{r}}{p}\right\}&
s_{\rho}&=\left\lfloor\frac{\frac{\ol{r}-1}{2}+\rho\ol{r}}{p}\right\rfloor& \textrm{if $\ol{r}\equiv 1\bmod 2$;}\\
r'_{\rho}&=p\left\{\frac{\frac{p+\ol{r}-1}{2}+\rho\ol{r}}{p}\right\}&
s'_{\rho}&=\left\lfloor\frac{\frac{p+\ol{r}-1}{2}+\rho\ol{r}}{p}\right\rfloor& \textrm{if $p+\ol{r}\equiv 1\bmod 2$}.
\end{align*}
The two conditions, $\ol{r}\equiv 1\bmod 2$ and $p+\ol{r}\equiv 1\bmod 2$, are not mutually exclusive; hence the quantities corresponding
to the two cases have different notation. Notice, though, that it cannot happen that $\ol{r}\equiv p+\ol{r}\equiv 0$ modulo 2, because
$\ol{r}$ and $p$ are coprime by definition.

\begin{lemma}\label{lemma1}
If, for some integers $t,u$, we have
\begin{equation}\label{eq:1}
\delta+\rho=tp+uq,
\end{equation}
then
\begin{equation}\label{eq:2}
u\equiv
\begin{cases}
r_\rho\bmod p,&\textrm{if $p\equiv \ol{r}\equiv 1\bmod 2$ or $p\equiv k+l\equiv 0\bmod 2$}\\
r_\rho'\bmod p,&\textrm{if $p\equiv 1$, $\ol{r}\equiv 0\bmod 2$ or $p\equiv 0$, $k+l\equiv 1\bmod 2$.}
\end{cases}
\end{equation}
\end{lemma}
\begin{proof}
Write \eqref{eq:1} as
\begin{equation}\label{eq:1a}\frac{(p-1)(q-1)}{2}+\rho=tp+u(kp+r).\end{equation}
If $p\equiv 1\bmod 2$, then \eqref{eq:1a} implies
\[\frac{(p-1)(r-1)}{2}+\rho\equiv ur\bmod p.\]
Multiply both sides by $\ol{r}$. Then the right hand side becomes $u\bmod p$, while the left hand side is $\frac{(p-1)(q-1)}{2}\ol{r}+\rho\ol{r}$.
Now 
\[\frac{\ol{r}(p-1)(r-1)}{2}=\frac{(p-1)(r\ol{r}-\ol{r})}{2}=\frac{(p-1)}{2}(1+lp-\ol{r})\equiv \frac{p-1}{2}(1-\ol{r})\bmod p.\]

Suppose $\ol{r}\equiv 1\bmod 2$. Then $\frac{p-1}{2}(1-\ol{r})\equiv \frac{\ol{r}-1}{2}\bmod p$.
Together this implies that
\[\frac{(p-1)(q-1)}{2}\ol{r}+\rho\ol{r}\equiv \frac{\ol{r}-1}{2}+\rho\ol{r}\bmod p\]
and the last expression is congruent to $r_\rho$ modulo $p$.

Suppose $\ol{r}\equiv 0\bmod 2$. We have that $\frac{(p-1)(1-\ol{r})}{2}=-p\frac{\ol{r}}{2}+\frac{\ol{r}+p-1}{2}$. The first term is congruent
to $0$ modulo $p$, hence $\frac{\ol{r}(p-1)(r-1)}{2}\equiv\frac{\ol{r}+p-1}{2}\bmod p$ and we conclude as in the previous case.

\smallskip
If $p\equiv 0\bmod 2$, then $\ol{r}\equiv r\equiv 1\bmod 2$ and \eqref{eq:1a} implies
\[\frac{1-kp-r}{2}+\rho\equiv ur\bmod p.\]
Multiplying both sides by $\ol{r}$ we obtain \eqref{eq:2}. 
\end{proof}
\begin{lemma}\label{lemma2}\
\begin{itemize}
\item If $\ol{r}\equiv 1\bmod 2$, the set of remainders $\rho\bmod p$ such that
\begin{equation}\label{eq:3}
r_\rho\le \rintfrac{p}{2} -1
\end{equation}
has exactly $\rintfrac{p}{2}$ elements including $0$.
\item If $p+\ol{r}\equiv 1\bmod 2$, the set of remainders $\rho\bmod p$ such that
\begin{equation}\label{eq:4}
r_\rho'\ge\rintfrac{p}{2}
\end{equation}
has exactly $\intfrac{p}{2}$ elements including $0$.
\end{itemize}
\end{lemma}
\begin{proof}
The conditions \eqref{eq:3} and \eqref{eq:4} are equivalent to congruences
\begin{align*}
\frac{\ol{r}-1}{2}+\rho\ol{r}&\equiv\tau\bmod p,\\
\frac{p+\ol{r}-1}{2}+\rho\ol{r}&\equiv\tau\bmod p,
\end{align*}
where $0\le\tau\le\rintfrac{p}{2}-1$ or $\rintfrac{p}{2}\le\tau<p$, respectively. Since $\gcd(\ol{r},p)=1$,
each of these congruences has, for a given $\tau$, exactly one solution $\rho$, which proves the lemma.
\end{proof}
\begin{lemma}\label{lemma3}
If $\ol{r}\equiv 1\bmod 2$, then $K(p,q)$ is the greatest non-negative integer $\sigma$ such that for all non-negative integers $\rho\le\sigma$
\begin{equation}\label{eq:5}
r_\rho\le\rintfrac{p}{2}-1.
\end{equation}
If $\ol{r}\equiv 0\bmod 2$, then $K(p,q)$ is the greatest non-negative integer $\sigma$ such that for all non-negative integers $\rho\le\sigma$
\begin{equation}\label{eq:6}
r'_\rho\ge\rintfrac{p}{2}.
\end{equation}
\end{lemma}
\begin{proof}
We consider the following four cases:
\begin{itemize}
\item[(A)] $p\equiv 1$, $\ol{r}\equiv 1\bmod 2$,
\item[(B)] $p\equiv 1$, $\ol{r}\equiv 0\bmod 2$,
\item[(C)] $p\equiv 0$, $k+l\equiv 0\bmod 2$,
\item[(D)] $p\equiv 0$, $k+l\equiv 1\bmod 2$.
\end{itemize}
We remark that  cases (A) and (C) correspond to the situation where $\delta\in S(p,q)$, while
cases (B) and (D) correspond to situation where $\delta\notin S(p,q)$

The proof of the four cases uses the following argument. We take $\rho\le\sigma$ and show that $\delta+\rho$ can (in cases (A) and (C)),
respectively cannot (in cases (B) and (D)) be presented as $tp+uq$ for $t,u\ge 0$. Next we show that $\sigma+1$ cannot (in cases (A) and (C)),
respectively can (in cases (B) and (D)) be presented in a similar manner. The first part shows that $K(p,q)\ge\sigma$ and the second part that
$K(p,q)<\sigma+1$. 

\medskip
\textbf{Case (A).} We have $\sigma\ge 0$ and by Lemma~\ref{lemma2}, $\sigma\le\frac{p-1}{2}$. 
First we are going to show that if $\rho\ge 0$, $r_\rho\le\rintfrac{p}{2}-1$, and $\rho\le\frac{p-1}{2}$, then $\delta+\rho$
belongs to the semigroup. We have
\begin{equation}\label{eq:caseA}
\delta+\rho=\frac{(p-1)(kp+r-1)}{2}+\rho=t_\rho p+r_\rho(kp+r),
\end{equation}
where 
\[t_\rho=k\left(\frac{p-1}{2}-r_\rho\right)+\frac{r-l-1}{2}+r s_\rho-\rho l.\]
Notice that as $r_\rho\in [0,p)$, we have that $\delta+\rho\in S(p,q)$ if and only if $t_\rho\ge 0$.

Since $r\equiv 1+l\bmod 2$, we have that $t_\rho\in\Z$. It follows from \eqref{eq:3} that
\[ps_\rho+\frac{p-1}{2}\ge \frac{\ol{r}-1}{2}+\rho\ol{r}.\]
Hence
\[p(s_\rho+\frac12)\ge \ol{r}(\rho+\frac12),\qquad s_\rho\ge\frac{\ol{r}(\rho+\frac12)}{p}-\frac12.\]
Therefore,
\begin{align*}
t_\rho &\ge k\left(\frac{p-1}{2}-r_\rho\right)+\frac{r-l-1}{2}+\frac{r\ol{r}(\rho+\frac12)}{p}-\frac{r}{2}-\rho l \\
&= k\left(\frac{p-1}{2}-r_\rho\right)-\frac{l+1}{2}+\frac{(1+pl)(\rho+\frac12)}{p}-\rho l \\
&\ge k\left(\frac{p-1}{2}-r_\rho\right)-\frac12 \\
&\ge -\frac12.
\end{align*}
As $t_\rho\in\Z$, we have that $t_\rho\ge 0$. Therefore $\delta+\rho\in S(p,q)$.

\smallskip
Consider now $\sigma+1$. If $\delta+\sigma+1\in S(p,q)$, then we can write
\[\delta+\sigma+1=tp+u(kp+r)\]
for some $t,u\ge 0$.
By the definition of $\sigma$, we have that
\[r_{\sigma+1}\ge\frac{p+1}{2},\qquad \sigma+1\le\frac{p+1}{2}.\]
Hence by Lemma~\ref{lemma1}, $u\ge\frac{p+1}{2}$. As  $t\ge 0$, we have
\[tp+u(kp+r)\ge \frac{p+1}{2}(kp+r)>\frac{(p-1)(kp+r-1)}{2}+\frac{p+1}{2}.\]
It follows that we cannot have $t,u\ge 0$; that is, $\sigma+1\notin S(p,q)$.
We conclude that
$K(p,q)=\sigma$.

\smallskip
\textbf{Case (B).} By Lemma~\ref{lemma2},  $\sigma\le \frac{p-3}{2}$. 
Suppose $\rho$ is such that $r'_\rho\ge\rintfrac{p}{2}$.
If we write
\[\delta+\rho=\frac{(p-1)(kp+r-1)}{2}+\rho=tp+u(kp+r),\]
where $t,u\in\Z_{\ge 0}$, then by Lemma~\ref{lemma1}
\[u\equiv r'_\rho\bmod p.\]
By \eqref{eq:4}, $u\ge\frac{p+1}{2}$ and for $t\ge 0$
\[\frac{(p-1)(kp+r-1)}{2}\le \frac{(p-1)(kp+r-1)}{2}+\frac{p-3}{2}<u(kp+r).\]
The contradiction shows that $\delta+\rho\notin S(p,q)$.

On the other hand, by the definition of $\sigma$ and by Lemma~\ref{lemma2}:
\begin{equation}\label{eq:5a}
r'_{\sigma+1}\le \frac{p-1}{2},\qquad \sigma+1\le\frac{p-1}{2}.
\end{equation}
Write $\delta+\sigma+1=t'_{\sigma+1}p+r'_{\sigma+1}(kp+r)$. We have
\[t'_{\sigma+1}=k\left(\frac{p-1}{2}-r'_{\sigma+1}\right)-\frac{l+1}{2}+(rs'_{\sigma+1}-(\sigma+1)l).\]
Since $l\equiv 1\bmod 2$, $t'_{\sigma+1}\in\Z$. Also \eqref{eq:5a} implies
\[ps'_{\sigma+1}+\frac{p-1}{2}\ge\frac{p+\ol{r}-1}{2}+(\sigma+1)\ol{r}.\]
Hence
\[s'_{\sigma+1}\ge \frac{(\sigma+\frac32)\ol{r}}{p}.\]
and
\begin{multline*}
t'_{\sigma+1}\ge k\left(\frac{p-1}{2}-r'_{\sigma+1}\right)-\frac{l+1}{2}+\frac{(\sigma+\frac32)(1+pl)}{p}-(\sigma+1)l=\\
k\left(\frac{p-1}{2}-r'_{\sigma+1}\right)-\frac{l+1}{2}+\frac{\sigma+\frac32}{p}+\frac{l}{2}\ge -\frac12.
\end{multline*}
Thus $t'_{\sigma+1}\ge 0$, so $\sigma+1\in S(p,q)$. Therefore $K(p,q)=\sigma$.

\smallskip
\textbf{Case (C).}
By Lemma~\ref{lemma2}, $\sigma\ge 0$ and  $\sigma\le \frac{p}{2}-1$. Choose $\rho$ such that $r_\rho\le\rintfrac{p}{2}-1$. Write
\[\delta+\rho=\frac{(p-1)(kp+r-1)}{2}+\rho=t_\rho p+r_{\rho}(kp+r),\]
where
\[t_\rho=k\left(\frac{p}{2}-r_\rho\right)+\frac{r-k-l-1}{2}+(rs_\rho-l\rho).\]
We will show that $t_\rho\ge 0$, which implies that $\delta+\rho\in S(p,q)$.
Since we have $r\equiv k+l+1\bmod 2$, it follows that $t_\rho\in\Z$. Also by \eqref{eq:3}:
\[s_\rho\ge \frac{\ol{r}(p+\frac12)}{p}-\frac12;\]
compare Case (A). Hence
\[t_\rho\ge k+\frac{r-k-l-1}{2}+\frac{(1+lp)(\rho+\frac12)}{p}-\frac{r}{2}-l\rho=\frac{k-l-1}{2}+\frac{\rho+\frac12}{p}+\frac{l}{2}>\frac{k-1}{2}.\]
Thus $t_\rho\ge 0$ and $\delta+\rho\in S(p,q)$. On the other hand, by the definition of $\sigma$ and by Lemma~\ref{lemma2}:
\[r_{\sigma+1}\ge \frac{p}{2},\qquad \sigma+1\le\frac{p}{2}.\]
Thus if
\[\frac{(p-1)(kp+r-1)}{2}+\sigma+1=tp+u(kp+r),\]
where $t,u\in\Z_{\ge 0}$, we have by Lemma~\ref{lemma1} that $u\ge r_{\sigma+1}\ge\frac{p}{2}$ and for $t\ge 0$
\[tp+u(kp+r)\ge u(kp+r)\ge\frac{p}{2}(kp+r)>\frac{(p-1)(kp+r-1)}{2}+\frac{p}{2}.\]
This contradicts the assumption that $\sigma+1\le\frac{p}{2}$, so we cannot have $t,u\in\Z_{\ge 0}$.
It follows that $\sigma+1\notin S(p,q)$, so $K(p,q)=\sigma$.

\smallskip
\textbf{Case (D).}
By Lemma~\ref{lemma2}, we have $\sigma\ge 0$ and $\sigma\le\frac{p}{2}-1$. 
%Here the conditions \eqref{eq:5} and \eqref{eq:6} are equivalent. \todo[color=red!10]{I don't understand what it means for conditions \eqref{eq:5} and \eqref{eq:6} to be equivalent. Perhaps it should be a different reference? \eqref{eq:3} and \eqref{eq:4}? } 
Write
\[\delta+\rho=tp+u(kp+r).\]
Suppose $\rho\le\sigma$.
By Lemma~\ref{lemma1}, $u\ge r'_{\rho}\ge\frac{p}{2}$ and for $t\ge 0$
\[tp+u(kp+r)\ge \frac{p}{2}(kp+r)>\frac{(p-1)(kp+r-1)}{2}+\frac{p}{2}>\delta+\rho.\]
Therefore $\delta+\rho$ cannot be presented as $tp+uq$ for $t,u\ge 0$, hence $\delta+\rho\notin S(p,q)$.
On the other hand by the definition of $\sigma$ and by Lemma~\ref{lemma2}:
\begin{equation}\label{eq:6a}
r'_{\sigma+1}<\frac{p}{2},\qquad \sigma+1\le\frac{p}{2}.
\end{equation}
Write
\[\delta+\sigma+1=t'_{\sigma+1}p+r'_{\sigma+1}(kp+r),\]
where
\[t'_{\sigma+1}=k\left(\frac{p}{2}-r'_{\sigma+1}\right)-\frac{k+l+1}{2}+rs'_{\sigma+1}-(\sigma+1)l.\]
Since $k+l\equiv 1\bmod 2$, we have that $t'_{\sigma+1}\in\Z$. Also \eqref{eq:6a} implies
\[ps'_{\sigma+1}+\frac{p-1}{2}\ge\frac{p+\ol{r}-1}{2}+(\sigma+1)\ol{r}.\]
Thus
\[s'_{\sigma+1}\ge\frac{(\sigma+\frac32)\ol{r}}{p}\]
and
\[t'_{\sigma+1}\ge k-\frac{k+l+1}{2}+\frac{(1+pl)(\sigma+\frac32)}{p}-(\sigma+1)l\ge -\frac12.\]
Therefore $t'_{\sigma+1}\ge 0$. That is, $\sigma+1\in S(p,q)$. We conclude that $K(p,q)=\sigma$. 
\end{proof}
\begin{lemma}\label{lemma4}
If $q>p+1$, then $K(p,q)=K(p,q-p)$.
\end{lemma}
\begin{proof}
By Lemma~\ref{lemma3}, the quantity $K(p,q)$ depends only on the residue $r$ of $q\bmod p$. Thus it does not change when $q$ is replaced
by $q-p$.
\end{proof}
\begin{proof}[Proof of Theorem~\ref{thm:sch1}]
We proceed by induction on $n$. For $n=1$, we have $\frac{q}{p}=a_0+\frac{1}{a_1}$, that is, $q=a_0p+1$ and $p=a_1$. 
In this case $r=\ol{r}=1$.  
By the definition of $r_\rho$, we have that $r_\rho=\rho\bmod p$.
Using Lemma~\ref{lemma3}, we obtain
$K(p,q)=\lceil\frac{p}{2}\rceil-1=\lfloor\frac{a_n-1}{2}\rfloor$. 

The inductive step is provided by Lemma~\ref{lemma4}. Namely, if $\frac{q}{p}=[a_0,\ldots,a_n]$, then we replace $q$ by $q'=q-a_0p$. By
Lemma~\ref{lemma4}, we have $K(p,q)=K(p,q')$. On the other hand, $\frac{q'}{p}=[a_1,\ldots,a_n]$. The induction step follows.
\end{proof}

\end{document}